\documentclass[11pt]{amsart}
\usepackage[centertags]{amsmath}
\usepackage{amsfonts}
\usepackage{amssymb}
\usepackage{amsthm}
\usepackage{newlfont}
\usepackage{amscd}
\usepackage{mathrsfs}
\usepackage[all,cmtip]{xy}
\usepackage{tikz}
\usepackage{cite}
\usepackage{fancyhdr}
\usepackage{stmaryrd}
\usepackage[pagebackref=false,colorlinks]{hyperref}

\definecolor{mycolor}{HTML}{F7F8E0}
\definecolor{myorange}{RGB}{245,156,74}
\definecolor{cadetgrey}{rgb}{0.57, 0.64, 0.69}
\definecolor{calpolypomonagreen}{rgb}{0.12, 0.3, 0.17}

\hypersetup{pdffitwindow=true,linkcolor=calpolypomonagreen,citecolor=calpolypomonagreen,urlcolor=calpolypomonagreen}

\usepackage[OT2,OT1]{fontenc}
\newcommand\cyr{%
\renewcommand\rmdefault{wncyr}%
\renewcommand\sfdefault{wncyss}%
\renewcommand\encodingdefault{OT2}%
\normalfont
\selectfont}
\DeclareTextFontCommand{\textcyr}{\cyr}

\newtheorem{thm}{Theorem}[section]
\newtheorem{cor}[thm]{Corollary}
\newtheorem{lem}[thm]{Lemma}
\newtheorem{prop}[thm]{Proposition}

\newtheorem{assu}[thm]{Assumption}
\newtheorem{conj}[thm]{Conjecture}

\theoremstyle{definition}

\newtheorem{rem}[thm]{Remark}

\newtheorem{exam}[thm]{Example}

 \numberwithin{equation}{section}
 
\newcommand{\Mod}[1]{\ (\text{mod}\ #1)}

\begin{document}
\title{On the refined conjectures on Fitting ideals of Selmer groups of elliptic curves with supersingular reduction}
\author{Chan-Ho Kim}
\address{(Chan-Ho Kim) School of Mathematics, Korea Institute for Advanced Study (KIAS), 85 Hoegiro, Dongdaemun-gu, Seoul 02455, Republic of Korea}
\email{chanho.math@gmail.com}
\author{Masato Kurihara}
\address{(Masato Kurihara) Department of Mathematics, Faculty of Science and Technology, Keio University, 3-14-1 Hiyoshi, Kohoku-ku, Yokohama, 223-8522, Japan}
\email{kurihara@math.keio.ac.jp}
\date{\today}
\subjclass[2010]{11R23 (Primary); 11G05, 11G40, 11F67 (Secondary)}
\keywords{Iwasawa theory, elliptic curves, Mazur-Tate conjectures, refined Iwasawa theory, supersingular primes}

\maketitle
\begin{abstract}
In this paper, we study the Fitting ideals of Selmer groups over finite subextensions in the cyclotomic $\mathbb{Z}_p$-extension of $\mathbb{Q}$ of an elliptic curve over $\mathbb{Q}$. 
Especially, we present a proof of the ``weak main conjecture" \`{a} la Mazur and Tate for elliptic curves with good (supersingular) reduction at an odd prime $p$.
 We also prove the ``strong main conjecture" suggested by the second named author under the validity of the $\pm$-main conjecture and the vanishing of a certain error term. 
The key idea is the explicit comparison among ``finite layer objects", 
``$\pm$-objects", and ``fine objects" in Iwasawa theory.
The case of good ordinary reduction is also treated.
\end{abstract}

\setcounter{tocdepth}{1}
\tableofcontents
\section{Introduction}
\subsection{Overview} \label{subsec:overview}
The main aim of this paper is to understand Selmer groups of an elliptic curve with supersingular reduction at $p$ over finite subextensions in the cyclotomic $\mathbb{Z}_p$-extension of $\mathbb{Q}$ by using $\pm$-Iwasawa theory \`{a} la Kobayashi-Pollack.
Let $E$ be an elliptic curve over $\mathbb{Q}$ with good reduction at an odd prime $p$. We assume that $a_p(E) \not\equiv 1 \Mod{p}$ throughout this article.

The $\pm$-Iwasawa theory is developed to understand Iwasawa theory for elliptic curves at supersingular primes (with assumption $a_p(E) = 0$).
In the supersingular setting, the usual Selmer groups over $\mathbb{Z}_p$-extensions and $p$-adic $L$-functions do not behave well as in the classical framework of Iwasawa theory.
Introducing $\pm$-Selmer groups and $\pm$-$p$-adic $L$-functions, Kobayashi \cite{kobayashi-thesis} and Pollack \cite{pollack-thesis} could apply the standard techniques of Iwasawa theory of elliptic curves with \emph{ordinary} reduction to the \emph{supersingular} setting.

On the other hand, Mazur-Tate conjectures \cite{mazur-tate} and the \emph{refined} Iwasawa theory \`{a} la the second named author (\cite{kurihara-invent}, \cite{kurihara-fitting}, and \cite{kurihara-munster}) focus on understanding Iwasawa theory over \emph{finite} abelian extensions over $\mathbb{Q}$.

In general, the refined Iwasawa theory (at finite layers) is regarded as a more delicate subject than the usual Iwasawa theory (at the infinite layer) since neither we can directly apply the theory of Iwasawa modules to finite layer objects nor we can ignore ``finite errors". It is well known that the structure of group rings at finite layers is much more complicated than that of the Iwasawa algebra.

In this article, we consider the subextensions in the cyclotomic $\mathbb{Z}_p$-extension whose Galois group is cyclic of $p$-power order and only one prime ramifies. Thus, this case can be regarded as the simplest one, but the full understanding of the following conjectures is still deep. Their precise formulations are given in $\S$\ref{subsec:conjectures}.
\begin{conj}[Mazur-Tate's weak main conjecture, Conjecture \ref{conj:mazur-tate-weak-main-conj}]
Assume that $E$ has no rational $p$-torsion.
Then the Mazur-Tate element of $E$ at a finite layer is contained in the Fitting ideal of the dual Selmer group of $E$ over the finite extension.
\end{conj}
\begin{conj}[The strong main conjecture, Conjecture \ref{conj:kurihara}]
Assume that $E$ has no rational $p$-torsion and $p$ does not divide the Tamagawa number of $E$.
Then the Mazur-Tate element of $E$ at a finite layer and the traces of the Mazur-Tate elements of $E$ at all the lower layers generate the Fitting ideal of the dual Selmer group of $E$ over the finite extension.
\end{conj}
In the case of good ordinary reduction with non-anomalous prime $p$ (i.e. $a_p(E) \not\equiv 1 \Mod{p}$), both conjectures follow from several standard ingredients in Iwasawa theory, including the Iwasawa main conjecture, the non-existence of proper $\Lambda$-submodules of finite index in the Selmer groups over the Iwasawa algebra $\Lambda$, and the control theorem. 
Although this case is more or less well-known to experts, the argument is not explicitly written in the literature.
Thus, we give a proof for the case of good ordinary reduction in $\S$\ref{sec:ordinary}.
We note that in this case the Fitting ideal of the Selmer group is principal.

In the case of good supersingular reduction, the situation becomes much more complicated.
Actually, the Fitting ideal of the Selmer group is never principal in this case.
Very fortunately, we are able to strengthen the argument of the good ordinary reduction case by making an explicit comparison between Selmer groups and $\pm$-Selmer groups in finite layers. This approach allows us to obtain the weak main conjecture. The proof is given in $\S$\ref{sec:putting}.
We obtain Theorem \ref{thm:main_theorem} in this way.

Concerning the strong main conjecture, we prove it in Theorem \ref{thm:main_theorem_3} under certain assumptions including the validity of the $\pm$-main conjecture.
We also provide many examples which satisfy these assumptions in Example \ref{exam:strong_main_conjecture}, so we have many examples for which the strong main conjecture holds.
Especially, if the fine Selmer group over the $\mathbb{Z}_p$-extension is ``all Mordell-Weil" described in Example \ref{exam:strong_main_conjecture}, then we can prove the strong main conjecture.
More generally, even without the assumptions imposed in Theorem \ref{thm:main_theorem_3}, we are able to prove a slightly weaker version of the strong main conjecture in Theorem \ref{thm:main_theorem_2}. In the weaker version, the statement involves an error term.

In the proof of Theorem \ref{thm:main_theorem_2}, we make an explicit comparison between Selmer groups and fine Selmer groups in finite layers.
This comparison is related to the finite layer version of the construction of algebraic $p$-adic $L$-functions \`{a} la Perrin-Riou. 
See \cite[2.4.3 Proposition]{perrin-riou-book}, \cite[$\S$3.1]{perrin-riou-supersingular}, for example.
The error term in Theorem \ref{thm:main_theorem_2} occurs in this finite layer comparison. 
Indeed, the assumptions in Theorem \ref{thm:main_theorem_3} are strong enough to force the error term to vanish.
As a result, we deduce a ``\emph{lower bound}" of Selmer groups over finite extensions from the Iwasawa main conjecture and some Fitting ideal techniques described in Appendix \ref{sec:fitting}. 
The proof of Theorem \ref{thm:main_theorem_2} is given in $\S$\ref{sec:non-CM} and the proof of Theorem \ref{thm:main_theorem_3} is given in $\S$\ref{sec:vanishing_coker_g_n}.

It seems that our approach does not work directly if $p \mid a_p(E)$ but $a_p(E) \neq 0$ since the $\sharp/\flat$-Iwasawa theory \`{a} la Sprung \cite{sprung-ap-nonzero} does not behave well in finite layers. See \cite[Open Problem 7.22]{sprung-ap-nonzero} for detail.

In the rest of this section, we introduce various conjectures we are interested in and state our main results and their applications.
In $\S$\ref{sec:ordinary}, we review the case for elliptic curves with good ordinary reduction and give a proof of the weak and strong main conjectures for this case in Theorem \ref{thm:main_theorem}.
In $\S$\ref{sec:tools}, we review relevant $\pm$-Iwasawa theory for elliptic curves.
In $\S$\ref{sec:putting}, we prove the weak main conjecture for elliptic curves with supersingular reduction in Theorem \ref{thm:main_theorem}.
In $\S$\ref{sec:non-CM}, we prove the slightly weaker version of the strong main conjecture for elliptic curves in Theorem \ref{thm:main_theorem_2}.
In $\S$\ref{sec:vanishing_coker_g_n}, we prove the strong main conjecture for elliptic curves under certain assumptions in Theorem \ref{thm:main_theorem_3}.
In Appendix \ref{sec:fitting}, we study refined techniques on Fitting ideals.
\subsection{Conjectures} \label{subsec:conjectures}
We recall various conjectures on the arithmetic of elliptic curves.
\subsubsection{Birch and Swinnerton-Dyer conjecture}
One of the leading problems of modern number theory is the following conjecture.
\begin{conj}[Birch and Swinnerton-Dyer] \label{conj:bsd}
Let $E$ be an elliptic curve over $\mathbb{Q}$. Then 
$$\mathrm{rk}_{\mathbb{Z}} E(\mathbb{Q}) = \mathrm{ord}_{s=1} L(E,s).$$
\end{conj}
We recall the formulation of the refinements and variants of Conjecture \ref{conj:bsd}.
\subsubsection{Setting the stage} \label{subsubsec:setting}
Let $p$ be an odd prime. Fix embeddings $\iota_p: \overline{\mathbb{Q}} \hookrightarrow \overline{\mathbb{Q}}_p$ and
$\iota_\infty: \overline{\mathbb{Q}} \hookrightarrow \mathbb{C}$.
Let $E$ be an elliptic curve over $\mathbb{Q}$ of conductor $N$ with $(N,p) = 1$.
Let
$$\overline{\rho} : \mathrm{Gal}(\overline{\mathbb{Q}}/\mathbb{Q}) \to \mathrm{Aut}_{\mathbb{F}_p}(E[p]) \simeq \mathrm{GL}_2(\mathbb{F}_p)$$
be the mod $p$ representation arising from the $p$-torsion points on $E$. 
For a prime $\ell$ dividing $N$, let $E_0(\mathbb{Q}_\ell)$ be the preimage of the nonsingular locus of $\widetilde{E}(\mathbb{F}_\ell)$. 
Then the \textbf{Tamagawa number of $E$} is defined by
 ${ \displaystyle \mathrm{Tam}(E) := \prod_{\ell \mid N} c_\ell }$  where $c_\ell = [E(\mathbb{Q}_\ell):E_0(\mathbb{Q}_\ell)]$.
 
Let $n \geq 1$ be an integer and $\mathbb{Q}_n$ the subextension of $\mathbb{Q}$ in $\mathbb{Q}(\mu_{p^{n+1}})$ with $\mathrm{Gal}(\mathbb{Q}_n/\mathbb{Q}) \simeq \mathbb{Z}/p^n\mathbb{Z}$. Let $\mathbb{Q}_\infty = \bigcup_{n\geq 1} \mathbb{Q}_n$ be the cyclotomic $\mathbb{Z}_p$-extension of $\mathbb{Q}$.
Let $\Gamma_n := \mathrm{Gal}(\mathbb{Q}_\infty/\mathbb{Q}_n)$ and $\Gamma := \mathrm{Gal}(\mathbb{Q}_\infty/\mathbb{Q})$.
Let $\Lambda_n := \mathbb{Z}_p[\mathrm{Gal}(\mathbb{Q}_n/\mathbb{Q})] = \mathbb{Z}_p[\Gamma/\Gamma_n]$ and $\Lambda := \varprojlim_n \Lambda_n = \mathbb{Z}_p \llbracket \mathrm{Gal}(\mathbb{Q}_\infty/\mathbb{Q}) \rrbracket = \mathbb{Z}_p \llbracket \Gamma \rrbracket $.
Let $\omega_n = \omega_n(X) := (1+X)^{p^n} - 1$.
Fix a generator $\gamma$ of $\mathrm{Gal}(\mathbb{Q}_\infty/\mathbb{Q})$ and take a generator $\gamma_n$ of $\mathrm{Gal}(\mathbb{Q}_n/\mathbb{Q})$ as the image of $\gamma$. Then we have isomorphisms
\[
\xymatrix{
\Lambda_n  \simeq \mathbb{Z}_p[X]/ \left( \omega_n(X) \right), &
\Lambda  \simeq \mathbb{Z}_p\llbracket X \rrbracket
}
\]
by sending the generators to $1+X$. Via the latter isomorphism, we also regard $\omega_n \in \Lambda$.

Let $\Phi_n(1+X) = \omega_n / \omega_{n-1}$ where $\Phi_n$ is the $p^n$-th cyclotomic polynomial.
Let $\omega^{\pm}_0(X)  := X$, $\widetilde{\omega}^{\pm}_0(X)  := 1$, and
\[
\xymatrix@R=0em{
{\displaystyle \omega^{+}_n = \omega^{+}_n(X) := X\cdot \prod_{2 \leq m \leq n, m: \textrm{ even}}\Phi_m(1+X) } , &
{\displaystyle \omega^{-}_n = \omega^{-}_n(X)  := X\cdot \prod_{1 \leq m \leq n, m: \textrm{ odd}}\Phi_m(1+X) } , \\
{\displaystyle \widetilde{\omega}^{+}_n = \widetilde{\omega}^{+}_n(X)  := \prod_{2 \leq m \leq n, m: \textrm{ even}}\Phi_m(1+X) } , &
{\displaystyle \widetilde{\omega}^{-}_n = \widetilde{\omega}^{-}_n(X)  := \prod_{1 \leq m \leq n, m: \textrm{ odd}}\Phi_m(1+X) } .
}
\]
Then we have $\omega_n(X) = \omega^{\pm}_n(X) \cdot \widetilde{\omega}^{\mp}_n(X)$, respectively. We also regard
$\omega^{\pm}_n$, $\widetilde{\omega}^{\pm}_n$ as elements in $\Lambda_n$ or $\Lambda$ via fixed isomorphisms.
Also, we identify $\Lambda_n = \mathbb{Z}_p[\mathrm{Gal}(\mathbb{Q}_n/\mathbb{Q})] \simeq \mathbb{Z}_p[\mathrm{Gal}(\mathbb{Q}_{n,p}/\mathbb{Q}_p)]$ if necessary. Here, $\mathbb{Q}_{n,p}$ is the completion of $\mathbb{Q}_{n}$ at the prime above $p$.

Let $f \in S_2(\Gamma_0(N))$ be the newform attached to $E$ by \cite[Theorem A]{bcdt}.
Let $G'_{n+1} := \mathrm{Gal}(\mathbb{Q}(\mu_{p^{n+1}})/\mathbb{Q})/\lbrace \pm 1 \rbrace \simeq \left(\mathbb{Z}/p^{n+1}\mathbb{Z}\right)^\times /\lbrace \pm 1 \rbrace $ and denote by $\sigma_a$ the element corresponding to $a \in \left(\mathbb{Z}/p^{n+1}\mathbb{Z}\right)^\times /\lbrace \pm 1 \rbrace$.
We define
 $$\theta'_{n+1}(f) := \sum_{ a \in \left(\mathbb{Z}/p^{n+1}\mathbb{Z}\right)^\times /\lbrace \pm 1 \rbrace } \left[\dfrac{a}{p^{n+1}}\right]^+ \cdot \sigma_a \in \mathbb{Z}_p[G'_{n+1}] .$$
Here, 
 $\left[\dfrac{a}{b}\right]^+$ is defined by
$$ 2 \pi \int^{\infty}_0 f \left( \frac{a}{b} + iy \right) dy = \left[\dfrac{a}{b}\right]^+ \cdot \Omega^+_E + 
 \left[\dfrac{a}{b}\right]^- \cdot \Omega^-_E$$
 where $\Omega^{\pm}_E$ are the N\'{e}ron periods of $E$. We write $\Omega_E = \Omega^+_E$.
The \textbf{Mazur-Tate element $\theta_n(f)$ of $f$ at $\mathbb{Q}_n$} is defined by the image of $\theta'_{n+1}(f)$ in $\Lambda_n$.
For simplicity, we assume that $\overline{\rho}$ is irreducible, and then we do not have to care about the integrality of Mazur-Tate elements and the Manin constant issue. See \cite[page 200--201]{kurihara-invent}.

We also define
$$\widetilde{\delta}_n := \sum_{a \in (\mathbb{Z}/n\mathbb{Z})^\times} \bigg( \overline{ \left[ \frac{a}{n} \right]^+} \cdot  \prod_{\ell \vert n} \overline{ \mathrm{log}_{\mathbb{F}_\ell} (a) }  \bigg) \in \mathbb{F}_p$$
where $n$ is the square-free product of Kolyvagin primes, $\overline{ \left[ \frac{a}{n} \right]^+}$ is the mod $p$ reduction of $\left[ \frac{a}{n} \right]^+$, and $\overline{ \mathrm{log}_{\mathbb{F}_\ell} (a) }$ is the mod $p$ reduction of the discrete logarithm of $a$ modulo $\ell$ (with a fixed primitive root modulo $\ell$, indeed).
Here, a prime $\ell$ is a Kolyvagin prime if $(\ell ,Np) =1$, $\ell \equiv 1 \Mod{p}$, and $a_\ell(E) \equiv \ell + 1 \Mod{p}$.
These $\widetilde{\delta}_n$'s were used to study the structure of Selmer groups in \cite{kurihara-iwasawa-2012}.
In addition, the non-vanishing of $\widetilde{\delta}_n$ for some $n$ implies the Iwasawa main conjecture for elliptic curves with any type of good reduction (\cite[Theorem 1.1]{kks}). 

Let $\Sigma$ be a finite set of places of $\mathbb{Q}$ including $p$, $\infty$, and the bad reduction primes of $E$, and $\mathbb{Q}_\Sigma$ be the maximal extension of $\mathbb{Q}$ unramified outside $\Sigma$.
We define the \textbf{Selmer group of $E$ over $\mathbb{Q}_n$} by
$$\mathrm{Sel}(\mathbb{Q}_n, E[p^\infty]) : =\mathrm{ker} \left( 
\mathrm{H}^1(\mathbb{Q}_\Sigma/\mathbb{Q}_n, E[p^\infty]) \to 
\prod_{v} \dfrac{\mathrm{H}^1(\mathbb{Q}_{n,v}, E[p^\infty])}{E(\mathbb{Q}_{n,v}) \otimes \mathbb{Q}_p/\mathbb{Z}_p}
\right) $$
where
$\mathrm{H}^1(\mathbb{Q}_\Sigma/\mathbb{Q}_n, E[p^\infty]) := \mathrm{H}^1(\mathrm{Gal}(\mathbb{Q}_\Sigma/\mathbb{Q}_n), E[p^\infty])$ is the Galois cohomology group,
 $v$ runs over all the (finite) places of $\mathbb{Q}_n$ dividing the places in $\Sigma$, and $\mathbb{Q}_{n,v}$ is the completion of $\mathbb{Q}_n$ at $v$.
We also define
the \textbf{Selmer group of $E$ over $\mathbb{Q}_\infty$} by
$$\mathrm{Sel}(\mathbb{Q}_\infty, E[p^\infty]) := \varinjlim_n \mathrm{Sel}(\mathbb{Q}_n, E[p^\infty]) .$$ 
It is well known that these Selmer groups are independent of the choice of $\Sigma$ (\cite[Corollary I.6.6]{milne-adt}).

We recall the notion of Fitting ideals for the convenience of readers.
For a ring $R$ and a finitely presented $R$-module $M$, take a presentation 
\[
\xymatrix{
R^s \ar[r]^-{h} & R^r \ar[r] & M \ar[r] & 0 
}
\]
where $h \in \mathrm{M}_{r \times s}(R)$.
Then the  
\textbf{Fitting ideal $\mathrm{Fitt}_R(M)$ of $M$ over $R$} is defined to be the ideal of $R$ generated by the determinants of the $r \times r$-minors of the matrix $h$.
It is well known that the Fitting ideal is independent of the choice of a presentation of $M$.

For a $\mathbb{Z}_p$-module $M$, let $M^\vee := \mathrm{Hom}_{\mathbb{Z}_p}(M, \mathbb{Q}_p/\mathbb{Z}_p)$.

\subsubsection{Mazur-Tate's refined conjecture}
In \cite{mazur-tate}, Mazur and Tate gave the following conjecture, which implies Conjecture \ref{conj:mazur-tate-weak-vanishing}. Conjecture \ref{conj:mazur-tate-weak-vanishing} is a refinement of the Birch and Swinnerton-Dyer conjecture (Conjecture \ref{conj:bsd}) in some sense. As we said in $\S$\ref{subsec:overview}, $a_p(E) \not\equiv 1 \Mod{p}$ is always assumed.
\begin{conj}[{\cite[Conjecture 3, ``weak main conjecture"]{mazur-tate}}] \label{conj:mazur-tate-weak-main-conj}
$$\theta_n(f) \in \mathrm{Fitt}_{\Lambda_n}\left( \mathrm{Sel}(\mathbb{Q}_n, E[p^\infty])^\vee \right) .$$
\end{conj}
\begin{rem}
Note that the original statement covers general abelian extensions of $\mathbb{Q}$ as we mentioned in $\S$\ref{subsec:overview}.
There are other approaches towards Conjecture \ref{conj:mazur-tate-weak-main-conj} due to 
Bley-Macias Castillo \cite[Theorem 2.12]{bley-castillo} assuming the $p$-part of the relevant equivariant Tamagawa number conjecture, 
Emerton-Pollack-Weston \cite{epw2} using the $p$-adic local Langlands correspondence as well as Kato's zeta elements, and
C.~Popescu using the theory of 1-motives.
We are informed that T. Kataoka proved the weak main conjecture over more general abelian extensions under certain assumptions by developing equivariant $\pm$-Iwasawa theory for elliptic curves and by adapting our strategy in his Ph.D. thesis.
\end{rem}
\subsubsection{The (refined)$^2$ conjecture}
Comparing with Conjecture \ref{conj:mazur-tate-weak-main-conj}, the second named author proposed the following more refined conjecture, which we call the ``strong main conjecture". (cf. \cite[Remark after Conjecture 3]{mazur-tate}.)
This conjecture can be regarded as a refinement of the Iwasawa main conjecture.  As we said in $\S$\ref{subsec:overview}, $a_p(E) \not\equiv 1 \Mod{p}$ is always assumed.
\begin{conj}[{\cite[Conjecture 0.3, ``strong main conjecture"]{kurihara-invent}}] \label{conj:kurihara}
Let $E$ be an elliptic curve over $\mathbb{Q}$ with good reduction at an odd prime $p$.
If $E(\mathbb{Q})[p]$ is trivial and $p \nmid \mathrm{Tam}(E)$, then
$$\left( \theta_n(f), \nu_{n-1,n}\left(  \theta_{n-1}(f) \right) \right) = \mathrm{Fitt}_{\Lambda_n}\left( \mathrm{Sel}(\mathbb{Q}_n, E[p^\infty])^\vee \right) .$$
Here, $\nu_{n-1,n}$ is the trace map $\Lambda_{n-1} \to \Lambda_n$ defined by $\sigma \mapsto \sum_{\tau \mapsto \sigma} \tau$ for $\sigma \in \mathrm{Gal}(\mathbb{Q}_{n-1}/\mathbb{Q})$ where $\tau$ runs over all elements of $\mathrm{Gal}(\mathbb{Q}_{n}/\mathbb{Q})$ projecting to $\sigma$.
\end{conj}
\begin{rem}
This conjecture explains the growth of $\mathrm{Sel}(\mathbb{Q}_n, E[p^\infty])$ as $n$ goes to infinity.
\end{rem}
The second named author proved Conjecture \ref{conj:kurihara} for the ``most basic" case (cf. \cite[$\S$5]{iovita-pollack}) using Kato's zeta elements.
\begin{thm}[{\cite[Theorem 0.1.(4)]{kurihara-invent}}] \label{thm:kurihara-most-basic}
If we further assume
\begin{enumerate}
\item $E$ has good supersingular reduction at $p$,
\item $p$ does not divide $\dfrac{L(E,1)}{\Omega_E}$, and
\item $\overline{\rho}$ is surjective
\end{enumerate}
as well as the assumptions of Conjecture \ref{conj:kurihara}, then Conjecture \ref{conj:kurihara} is true.
\end{thm}
\begin{rem}
Note that $a_p(E) = 0$ is not assumed in Theorem \ref{thm:kurihara-most-basic}.
In \cite[Theorem 1.1.(3)]{pollack-algebraic}, R.\@ Pollack proved an algebraic analogue of Theorem \ref{thm:kurihara-most-basic} using a formal group argument assuming $a_p(E) = 0$. His work does not require the surjectivity of $\overline{\rho}$.
\end{rem}
\begin{rem}
For the case of $p=2$, Conjecture \ref{conj:kurihara} may not hold. See \cite[Remark 1.2]{pollack-algebraic} and \cite{kurihara-otsuki} for detail.
\end{rem}

Pollack reformulates Conjecture \ref{conj:kurihara} in terms of his signed $p$-adic $L$-functions under the assumption $a_p(E) = 0$.
See $\S$\ref{subsec:basic_objects} for the characterization of the $\pm$-$p$-adic $L$-functions $L^{\pm}_p(\mathbb{Q}_\infty,f)$.
We recall a proposition of R.\@ Pollack, which shows us the connection between Mazur-Tate elements and $\pm$-$p$-adic $L$-functions.
\begin{prop}[{\cite[Proposition 6.18]{pollack-thesis}}] \label{prop:pollack_theta_pm}
$$\theta_n(f) \equiv \widetilde{\omega}^{\mp}_n \cdot L^{\mp}_p(\mathbb{Q}_\infty,f) \Mod{\omega_n}$$
in $\Lambda_n$ if $n$ is even/odd, respectively.
\end{prop}
Then, as ideals of $\Lambda$, the following equality holds
$$\left( \omega_n, \theta_n(f), \nu_{n-1,n} \left( \theta_{n-1}(f)\right) \right) =\left( \omega_n, \widetilde{\omega}^+_n \cdot L^+_p(\mathbb{Q}_\infty,f), \widetilde{\omega}^-_n \cdot L^-_p(\mathbb{Q}_\infty,f) \right) .$$
Thus, assuming $a_p(E) = 0$, Conjecture \ref{conj:kurihara} is equivalent to the following conjecture.
\begin{conj}[{\cite[Conjecture 6.19]{pollack-thesis}}] \label{conj:kurihara-pm}
We assume $a_p(E) = 0$ as well as the conditions in Conjecture \ref{conj:kurihara}.
Then 
$$\left( \widetilde{\omega}^+_n \cdot L^+_p(\mathbb{Q}_\infty,f) \Mod{\omega_n} , \widetilde{\omega}^-_n \cdot L^-_p(\mathbb{Q}_\infty,f) \Mod{\omega_n} \right) = \mathrm{Fitt}_{\Lambda_n} \left( \mathrm{Sel}(\mathbb{Q}_n, E[p^\infty])^\vee \right) .$$
\end{conj}
\begin{rem}
It is known that Kato's zeta elements exist integrally when $\overline{\rho}$ is surjective. In this case, Proposition \ref{prop:pollack_theta_pm} can be interpreted as a comparison between the $P_n$-pairing  made by the second named author and the $\pm$-Coleman maps made by Kobayashi (modulo $\omega_n$) as in the following diagram. See also \cite[$\S$1]{kurihara-pollack}.
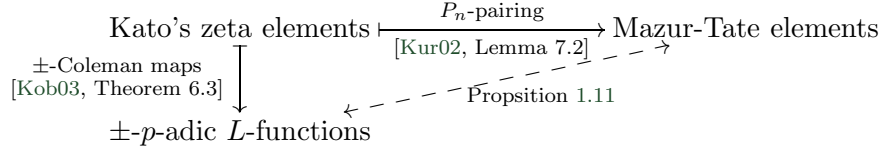
\[
\xymatrix{
\textrm{Kato's zeta elements} \ar@{|->}[rrr]^-{P_n\textrm{-pairing}}_-{\textrm{ \cite[Lemma 7.2]{kurihara-invent} }} \ar@{|->}[d]_-{ \substack{\textrm{$\pm$-Coleman maps} \\  \textrm{\cite[Theorem 6.3]{kobayashi-thesis} }} } & & & \textrm{Mazur-Tate elements} \ar@{<-->}[llld]^-{ \quad \textrm{Propsition \ref{prop:pollack_theta_pm}}} \\
 \textrm{$\pm$-$p$-adic $L$-functions}
}
\]
\end{rem}
\subsection{Main theorems}
We state three main theorems (mainly for elliptic curves with good supersingular reduction).
\begin{thm}[Main Theorem I] \label{thm:main_theorem}
Let $E$ be an elliptic curve over $\mathbb{Q}$ with good reduction at an odd prime $p$.
Assume that $\overline{\rho}$ is surjective if $E$ is non-CM.
Assume one of the followings:
\begin{enumerate}
\item[(ord)] If $p \nmid a_p(E)$, then $a_p(E) \not\equiv 1 \Mod{p}$, or 
\item[(ss)]  If $p \mid a_p(E)$, then $a_p(E) = 0$.
\end{enumerate}
Then
$$\left( \theta_n(f), \nu_{n-1,n} \left(  \theta_{n-1}(f) \right) \right) \subseteq \mathrm{Fitt}_{\Lambda_n}\left( \mathrm{Sel}(\mathbb{Q}_n, E[p^\infty])^\vee \right).$$
Thus, Mazur-Tate's weak main conjecture (Conjecture \ref{conj:mazur-tate-weak-main-conj}) for $E$ over $\mathbb{Q}_n$ follows.

In Case (ord), if we further assume $p \nmid \mathrm{Tam}(E)$ and the Iwasawa main conjecture (Conjecture \ref{conj:main-conj})
 holds, then the inclusion becomes an equality, so the strong main conjecture (Conjecture \ref{conj:kurihara}) holds.
\end{thm}
See $\S$\ref{sec:ordinary} for proof of Case (ord) and $\S$\ref{sec:putting} for proof of Case (ss). See Theorem \ref{thm:main-conj} and Theorem \ref{thm:main-conj-2} for the current status of the Iwasawa main conjecture.

Let $\chi : \mathrm{Gal}(\mathbb{Q}_n/\mathbb{Q}) \to \overline{\mathbb{Q}}^\times_p$ be a character and $\mathbb{Z}_p[\chi]$ the ring generated by the image of $\chi$ over $\mathbb{Z}_p$.
The map $\chi$ naturally extends to an algebra homomorphism 
$\mathbb{Z}_p[\mathrm{Gal}(\mathbb{Q}_n/\mathbb{Q})] \to \mathbb{Z}_p[\chi]$ defined by
$\sigma \mapsto \chi(\sigma)$
where $\sigma \in \mathrm{Gal}(\mathbb{Q}_n/\mathbb{Q})$ and also denote it by $\chi$.
Then we also define the \textbf{augmentation ideal at $\chi$} by
$$I_\chi := \mathrm{ker} \left( \chi: \Lambda_n \to \mathbb{Z}_p[\chi] \right). $$

Let $L \in \Lambda_n$. We say $L$ \textbf{vanishes to infinite order at $\chi$} if $L$ is contained in all powers of $I_\chi$. We say $L$ \textbf{vanishes to order $r$ at $\chi$} if $L \in I^r_\chi \setminus I^{r+1}_\chi$. See \cite[(1.5)]{mazur-tate} for detail.
\begin{conj}[{\cite[Conjecture 1, ``weak vanishing conjecture"]{mazur-tate}}] \label{conj:mazur-tate-weak-vanishing}
The order of vanishing of $\theta_n(f)$ at $\chi$ is greater than or equal to the dimension of the $\chi$-part of the Mordell-Weil group of $E(\mathbb{Q}_n)$.
\end{conj}

\begin{cor} \label{cor:mazur-tate-weak-vanishing}
Under the same assumptions of Theorem \ref{thm:main_theorem}, Conjecture \ref{conj:mazur-tate-weak-vanishing} holds.
\end{cor}
\begin{proof}
\cite[Proposition 3]{mazur-tate}.
\end{proof}
\begin{rem} $ $
\begin{enumerate}
\item In both conditions (ord) and (ss) in Theorem \ref{thm:main_theorem}, $a_p(E) \not\equiv 1 \Mod{p}$ or $a_p(E) = 0$ ensures that $E(\mathbb{Q})[p]$ is trivial.
\item An anticyclotomic analogue of Theorem \ref{thm:main_theorem} is investigated in \cite{kim-mazur-tate}.
\item See \cite{ota-thesis} for progress towards Conjecture \ref{conj:mazur-tate-weak-vanishing}.
\end{enumerate}
\end{rem}

In the case of elliptic curves with good supersingular reduction (under the surjectivity of $\overline{\rho}$), we can also obtain a \emph{lower} bound of the Selmer groups as follows.
Let ${\displaystyle T = \varprojlim_{n}E[p^n] }$ be the $p$-adic Tate module of $E$, and
$\mathbb{H}^1_{\mathrm{glob}}(T)$ the global Iwasawa cohomology (defined in $\S$\ref{subsec:kato_main_conj_n_fine_selmer}).
We define the ``error term" by
$$\mathrm{Err}_n := \mathrm{coker} \left( \frac{ \mathrm{H}^1(\mathbb{Q}_{\Sigma}/\mathbb{Q}_{n}, T) }{ \mathrm{im} \ \mathbb{H}^1_{\mathrm{glob}}(T)}
\to \frac{ \mathrm{H}^1(\mathbb{Q}_{n,p}, T) }{ E( \mathbb{Q}_{n,p} ) \otimes \mathbb{Z}_p + \mathrm{im} \ \mathbb{H}^1_{\mathrm{glob}}(T)}  \right) $$
where $\mathrm{im} \ \mathbb{H}^1_{\mathrm{glob}}(T)$ is the image of $\mathbb{H}^1_{\mathrm{glob}}(T)$ in $A$ in the notation $\frac{A}{\mathrm{im} \ \mathbb{H}^1_{\mathrm{glob}}(T)}$.
The error term $\mathrm{Err}_n$ also naturally appears as the cokernel of a certain map $g_n$ (Remark \ref{rem:coker_g_n}), which is explicitly defined in (\ref{eqn:coker_g_n}) (cf. \cite[(10.36) and Proposition 10.6]{kobayashi-thesis}).
\begin{thm}[Main Theorem II] \label{thm:main_theorem_2}
Let $E$ be an elliptic curve over $\mathbb{Q}$ with good supersingular reduction at an odd prime $p$.
Assume that $a_p(E) = 0$, $\overline{\rho}$ is surjective, and $p \nmid \mathrm{Tam}(E)$.
If the $\pm$-main conjectures (Conjecture \ref{conj:pm-main-conj}) hold, then
$$\mathrm{Fitt}_{\Lambda_n}\left( \mathrm{Err}_n \right) \cdot \mathrm{Fitt}_{\Lambda_n}\left( \mathrm{Sel}(\mathbb{Q}_n, E[p^\infty])^\vee \right) \subseteq \left( \theta_n(f), \nu_{n-1,n} \left(  \theta_{n-1}(f) \right) \right) \subseteq \mathrm{Fitt}_{\Lambda_n}\left( \mathrm{Sel}(\mathbb{Q}_n, E[p^\infty])^\vee \right).$$
\end{thm}
See $\S$\ref{sec:non-CM} for proof.
Note that the surjectivity of $\overline{\rho}$ implies $E$ has no CM.
Also, see Theorem \ref{thm:pm-main-conj} and Remark \ref{rem:pm-main-conj} for the current status of the $\pm$-Iwasawa main conjecture.
Although $\mathrm{Err}_n$ might not be zero in general, if it vanishes, then Conjecture \ref{conj:kurihara} holds.
\begin{rem}
As a $\mathbb{Z}_p$-module, $\mathrm{Err}_n$ is finitely generated.
Also, $\mathrm{Err}_n$ stabilizes as $n >>0$.
See \cite[(Proof of) Proposition 10.6]{kobayashi-thesis}.
\end{rem}
We define the \textbf{fine Selmer groups of $E$ over $\mathbb{Q}_n$} by
 $$\mathrm{Sel}_0(\mathbb{Q}_n, E[p^\infty]) := \mathrm{ker}\left( \mathrm{Sel}(\mathbb{Q}_n, E[p^\infty]) \to \mathrm{H}^1(\mathbb{Q}_{n,p},  E[p^\infty]) \right)$$
 and ${\displaystyle \mathrm{Sel}_0(\mathbb{Q}_\infty, E[p^\infty]) := \varinjlim_n \mathrm{Sel}_0(\mathbb{Q}_n, E[p^\infty]) }$ as in 
\cite[Definition 4.1]{kurihara-invent} and \cite[$\S$0.3]{kurihara-pollack}.
\begin{thm}[Main Theorem III] \label{thm:main_theorem_3}
Under the assumptions of Theorem \ref{thm:main_theorem_2}, we further assume
\begin{enumerate}
\item[(fineNF)] $\mathrm{Sel}_{0}(\mathbb{Q}_{\infty}, E[p^{\infty}])^\vee$ has no nontrivial finite 
$\Lambda$-submodule, and
\item[(\textrm{{\cyr SH}})] if $\Phi_{n}(1+X)$ divides a generator of $\mathrm{char}_{\Lambda} \left( \mathrm{Sel}_{0}(\mathbb{Q}_{\infty}, E[p^{\infty}])^\vee \right)$, then $\mathrm{rk}_{\mathbb{Z}} E(\mathbb{Q}_{n})>\mathrm{rk}_{\mathbb{Z}} E(\mathbb{Q}_{n-1})$
(if $n=0$, then $\Phi_{0}(1+X)=X$ and this inequality means $\mathrm{rk}_{\mathbb{Z}} E(\mathbb{Q})>0$).
\end{enumerate}
Then $\mathrm{Err}_n$ vanishes. Therefore, the strong main conjecture (Conjecture \ref{conj:kurihara})
$$\left( \theta_n(f), \nu_{n-1,n} \left(  \theta_{n-1}(f) \right) \right) = \mathrm{Fitt}_{\Lambda_n}\left( \mathrm{Sel}(\mathbb{Q}_n, E[p^\infty])^\vee \right)$$
holds.
\end{thm}
See $\S$\ref{sec:vanishing_coker_g_n} for proof.
\begin{exam} \label{exam:strong_main_conjecture}
There are many examples satisfying Assumptions (fineNF) and (\textrm{{\cyr SH}}) in Theorem \ref{thm:main_theorem_3}.
\begin{enumerate}
\item
We note that Assumption (\textrm{{\cyr SH}}) is satisfied if at least one of the following conditions is satisfied:
\begin{enumerate}
\item  The characteristic ideal of $\mathrm{Sel}_{0}(\mathbb{Q}_\infty, E[p^\infty])^{\vee}$ is prime to $\omega_n$ for all $n$.
\item If $\Phi_n(1+X)$ divides a generator of the characteristic ideal of $\mathrm{Sel}_{0}(\mathbb{Q}_\infty, E[p^\infty])^{\vee}$, then $\textrm{{\cyr SH}}(E/\mathbb{Q}_n)[p^\infty]$ is finite for all $n$.
\end{enumerate}
In fact, the implication of Assumption (\textrm{{\cyr SH}}) from (a) is trivial, and that from (b) can be proved by the control theorem for fine Selmer groups. (cf. \cite[Lemma 4.2, Remark 4.4]{kurihara-invent}) By (b), we expect that Assumption (\textrm{{\cyr SH}}) always holds.
Therefore, the only essential condition in Theorem \ref{thm:main_theorem_3} is Assumption (fineNF).
\item
 If one of the following conditions occurs, then both Assumptions (fineNF) and (\textrm{{\cyr SH}}) follow:
\begin{itemize}
\item $\mathrm{Sel}_0(\mathbb{Q}, E[p^\infty])$ is trivial. (See \cite[Lemma 4.3]{kurihara-invent} and \cite{pollack-algebraic}.)
\item $\mathrm{Sel}_0(\mathbb{Q}_{\infty}, E[p^\infty])_{\Gamma_n}$ is trivial. (It is a weaker condition than the one above.)
\item $\mathrm{Sel}_0(\mathbb{Q}_{\infty}, E[p^\infty]) = \mathrm{ker} \left( E(\mathbb{Q}_{\infty}) \otimes \mathbb{Q}_p/\mathbb{Z}_p \to E(\mathbb{Q}_{\infty,p}) \otimes \mathbb{Q}_p/\mathbb{Z}_p \right)$.
\end{itemize}
Note that the first two cases never occur if $\mathrm{rk}_{\mathbb{Z}}E(\mathbb{Q})>1$. The fine Selmer group is said to be \textbf{all Mordell-Weil} if the last case holds. In the all Mordell-Weil case, $\mathrm{Sel}_0(\mathbb{Q}_{n}, E[p^\infty])^\vee$ is free over $\mathbb{Z}_p$, so Assumptions (fineNF) and (\textrm{{\cyr SH}}) follow.
\item
It is conjectured by Greenberg that the roots of a generator of $\mathrm{char}_{\Lambda}(\mathrm{Sel}_{0}(\mathbb{Q}_\infty, E[p^\infty])^\vee)$ are all of the form $\zeta-1$ where $\zeta$ is a $p$-power root of unity. See \cite[Problem 0.7]{kurihara-pollack} for detail.
\end{enumerate}
\end{exam}

\section{Review of the case for elliptic curves with good ordinary reduction} \label{sec:ordinary}
In this section, we prove Theorem \ref{thm:main_theorem} for elliptic curves with good ordinary reduction.
\subsection{Tools from Iwasawa theory}
Let $E$ be an elliptic curve over $\mathbb{Q}$ with good ordinary reduction at $p$.
We first recall the $\Lambda$-cotorsion property of Selmer groups.
See \cite[Theorem 4.4]{rubin-main-conj-img-quad-fields}, \cite[Theorem 1.5]{greenberg-lnm}, and \cite[Theorem 17.4.(1)]{kato-euler-systems} for detail.
\begin{thm} \label{thm:torsion-ordinary}
The Selmer group $\mathrm{Sel}(\mathbb{Q}_\infty, E[p^\infty] )$ is $\Lambda$-cotorsion.
\end{thm}

The following statement is the Iwasawa main conjecture for elliptic curves with ordinary reduction.
See \cite[$\S$1.(c)]{mazur-IMC-for-E}, \cite[Conjecture 1.13]{greenberg-lnm}, and \cite[Conjecture 17.6]{kato-euler-systems} for detail.
\begin{conj}[Iwasawa main conjecture] \label{conj:main-conj}
Let $p$ be an odd prime and $E$ an elliptic curve over $\mathbb{Q}$ with $p \nmid a_p(E)$.
Then 
$$ \left( L_p(\mathbb{Q}_\infty, f_\alpha) \right)  =  \mathrm{char}_\Lambda \left( \mathrm{Sel}(\mathbb{Q}_\infty, E[p^\infty])^\vee \right) $$
where $L_p(\mathbb{Q}_\infty, f_\alpha)$ is the $p$-adic $L$-function of the $p$-stabilized form $f_\alpha$ with the unit root $\alpha$.
\end{conj}

The following theorem is due to Rubin \cite[Theorem 12.3]{rubin-main-conj-cm} for the CM case and Kato \cite[Theorem 17.4.(3)]{kato-euler-systems} for the non-CM case.
\begin{thm} \label{thm:main-conj}
Let $p$ be an odd prime and $E$ be an elliptic curve over $\mathbb{Q}$ with $p \nmid a_p(E)$.
\begin{enumerate}
\item If $E$ has CM, then Conjecture \ref{conj:main-conj} holds.
\item If $E$ has no CM, then we assume $\overline{\rho}$ is surjective.
Then 
$$ \left( L_p(\mathbb{Q}_\infty, f_\alpha) \right)  \subseteq \mathrm{char}_\Lambda \left( \mathrm{Sel}(\mathbb{Q}_\infty, E[p^\infty])^\vee \right) .$$
\end{enumerate}
\end{thm}
For the non-CM case, we have the following theorem is due to Skinner-Urban \cite[Theorem 3.29]{skinner-urban}, X. Wan \cite[Theorem 4]{wan_hilbert}, and K-Kim-Sun \cite[Theorem 1.1]{kks}.
\begin{thm} \label{thm:main-conj-2}
Keep all the assumptions in Theorem \ref{thm:main-conj}.
\begin{itemize}
\item[\cite{skinner-urban}] If there exists a prime $q \Vert N$ such that $\overline{\rho}$ is ramified at $q$, then Conjecture \ref{conj:main-conj} holds.
\item[\cite{wan_hilbert}] If there exists a real quadratic field $F/\mathbb{Q}$ such that
\begin{itemize}
\item $p$ is unramified in $F$,
\item any prime $q$ dividing $N$ such that $q \equiv -1 \Mod{p}$ is inert in $F/\mathbb{Q}$, and any other prime dividing $N$ splits in $F/\mathbb{Q}$,
\item the canonical period of $f$ over $F$ is the square of its canonical period over $\mathbb{Q}$ up to a $p$-adic unit,
\end{itemize}
then Conjecture \ref{conj:main-conj} holds.
\item[\cite{kks}] If $\widetilde{\delta}_n \neq 0$ for some $n$ and ${ \displaystyle p \nmid \mathrm{Tam}(E) \cdot  \prod_{q \vert N_{\mathrm{sp}}} (q-1) \cdot \prod_{q' \vert N_{\mathrm{ns}}} (q' +1)  }$ 
where $N_{\mathrm{sp}}$ is the product of split multiplicative reduction primes of $E$ and
$N_{\mathrm{ns}}$ is the product of non-split multiplicative reduction primes of $E$, then Conjecture \ref{conj:main-conj} holds.
\end{itemize}
\end{thm}
The following theorem is due to Greenberg and Hachimori-Matsuno (\cite[Proposition 4.14]{greenberg-lnm}, \cite[Corollary]{hachimori-matsuno}).
\begin{thm} \label{thm:no-finite-ordinary}
The Selmer group
$\mathrm{Sel}(\mathbb{Q}_\infty, E[p^\infty])$ has no proper $\Lambda$-submodule of finite index.
\end{thm}
We recall the control theorem for $E$ over the cyclotomic $\mathbb{Z}_p$-extension of $\mathbb{Q}$.
\begin{thm} \label{thm:control-ordinary}
The restriction map
$$\mathrm{Sel}(\mathbb{Q}_n, E[p^\infty]) \to \mathrm{Sel}(\mathbb{Q}_\infty, E[p^\infty])[\omega_n]$$
is injective with the finite cokernel whose size is bounded independently of $n$.
If we further assume that $a_p(E) \not\equiv 1 \Mod{p}$ and $p \nmid \mathrm{Tam}(E)$,
then the restriction map is an isomorphism.
\end{thm}
\begin{proof}
See \cite[Proposition 3.7, Proposition 3.8 and Proposition 3.9]{greenberg-lnm}.
\end{proof}
\subsection{Proof of Theorem \ref{thm:main_theorem} for the case of good ordinary reduction}
This is basically obtained in \cite[Corollary 10.3]{kurihara-fitting}.
From Theorem \ref{thm:main-conj}, we have
$$\left( L_p(\mathbb{Q}_\infty, f_\alpha) \right)  \subseteq \mathrm{char}_\Lambda \left( \mathrm{Sel}(\mathbb{Q}_\infty, E[p^\infty])^\vee \right) .$$
By Theorem \ref{thm:torsion-ordinary} and Theorem \ref{thm:no-finite-ordinary}, characteristic ideals are equal to Fitting ideals via Lemma \ref{lem:char-to-fitt}; thus, we have
$$\left( L_p(\mathbb{Q}_\infty, f_\alpha) \right)  \subseteq \mathrm{Fitt}_\Lambda \left( \mathrm{Sel}(\mathbb{Q}_\infty, E[p^\infty])^\vee \right) .$$
Taking the quotient by $\omega_n$, we have
$$\left( \vartheta_n( f_\alpha) \right)  \subseteq \mathrm{Fitt}_{\Lambda_n} \left( \left(\mathrm{Sel}(\mathbb{Q}_\infty, E[p^\infty])[\omega_n]\right)^\vee \right) $$
where
$$\vartheta_n( f_\alpha) = \frac{1}{\alpha^n} \cdot \left(  \theta_n( f) - \frac{1}{\alpha} \cdot \nu_{n-1,n} \left( \theta_{n-1}( f) \right)  \right)$$
is the $p$-stabilized Mazur-Tate element with the unit root $\alpha$.
Using Theorem \ref{thm:control-ordinary} and Lemma \ref{lem:fitting_quotient}, we have
$$\left( \vartheta_n( f_\alpha) \right)  \subseteq \mathrm{Fitt}_{\Lambda_n} \left( \mathrm{Sel}(\mathbb{Q}_n, E[p^\infty])^\vee \right) .$$
Since $a_p(E) \not\equiv 1 \Mod{p}$, it is not difficult to observe 
$$\vartheta_n( f_\alpha) = u \cdot \theta_n( f)$$
for some $u \in \Lambda^\times_n$.
It shows that
$$\left( \theta_n(f), \nu_{n-1,n} \left(  \theta_{n-1}(f) \right) \right) \subseteq \mathrm{Fitt}_{\Lambda_n}\left( \mathrm{Sel}(\mathbb{Q}_n, E[p^\infty])^\vee \right).$$
We also note that $\nu_{n-1,n}(\theta_{n-1}(f))$ is a multiple of $\vartheta_n(f)$, so the ideal $(\theta_n(f),\nu_{n-1,n}(\theta_{n-1}(f))$ is the principal ideal generated by $\vartheta_n(f)$, equivalently by $\theta_n(f)$.

If we further assume $p \nmid \mathrm{Tam}(E)$ and the Iwasawa main conjecture (Conjecture \ref{conj:main-conj}, Theorem \ref{thm:main-conj}, and Theorem \ref{thm:main-conj-2}), then all the inclusions in the proof become equalities, so we have
$$\left( \theta_n(f) \right) = \mathrm{Fitt}_{\Lambda_n}\left( \mathrm{Sel}(\mathbb{Q}_n, E[p^\infty])^\vee \right).$$

\section{Tools from {$\pm$}-Iwasawa theory} \label{sec:tools}

\subsection{Basic objects of $\pm$-Iwasawa theory} \label{subsec:basic_objects}
We quickly recall the basic objects of $\pm$-Iwasawa theory. For a more detailed description, we refer to \cite{kobayashi-thesis} for the algebraic side and to \cite{pollack-thesis} for the analytic side.

\begin{rem}[Sign convention]
We fix the sign convention of $\pm$-Iwasawa theory as follows:
\begin{enumerate}
\item Selmer groups: \cite{kobayashi-thesis}
\item $p$-adic $L$-functions: \cite{pollack-thesis} = $-$\cite{kobayashi-thesis}
\item Coleman maps: \cite{kurihara-pollack} = $-$\cite{kobayashi-thesis}
\end{enumerate}
\end{rem}

\subsubsection{Local conditions at $p$} \label{subsubsec:local_condition_at_p}
Let $E$ be an elliptic curve over $\mathbb{Q}$ with $a_p(E) = 0$.
Then we define 
\begin{align*}
E^+(\mathbb{Q}_{n,p}) & := \lbrace P \in E(\mathbb{Q}_{n,p}) : \mathrm{Tr}_{n/m+1} (P) \in E(\mathbb{Q}_{m,p}) \textrm{ for even } m \ (0 \leq m < n)\rbrace \\
E^-(\mathbb{Q}_{n,p}) & := \lbrace P \in E(\mathbb{Q}_{n,p}) : \mathrm{Tr}_{n/m+1} (P) \in E(\mathbb{Q}_{m,p}) \textrm{ for odd } m \ (0 \leq m < n)\rbrace
\end{align*}
where $\mathbb{Q}_{n,p}$ is the completion of $\mathbb{Q}_{n}$ at $p$ and $\mathrm{Tr}_{n/m+1} : E(\mathbb{Q}_{n,p}) \to E(\mathbb{Q}_{m+1,p})$ is the trace map.
\subsubsection{The norm subgroups}
Let $\widehat{E}$ be the formal group associated to $E$ and $\mathfrak{m}_n$ be the maximal ideal of $\mathbb{Q}_{n,p}$.
We define
\begin{align*}
\widehat{E}^+(\mathfrak{m}_{n}) & := \lbrace P \in \widehat{E}(\mathfrak{m}_{n}) : \mathrm{Tr}_{n/m+1} (P) \in \widehat{E}(\mathfrak{m}_{m}) \textrm{ for even } m \ (0 \leq m < n)\rbrace \\
\widehat{E}^-(\mathfrak{m}_{n}) & := \lbrace P \in \widehat{E}(\mathfrak{m}_{n}) : \mathrm{Tr}_{n/m+1} (P) \in \widehat{E}(\mathfrak{m}_{m}) \textrm{ for odd } m \ (0 \leq m < n)\rbrace
\end{align*}
where $\mathrm{Tr}_{n/m+1} : \widehat{E}(\mathfrak{m}_{n}) \to \widehat{E}(\mathfrak{m}_{m+1})$ is the trace map.
\subsubsection{$\pm$-Selmer groups}
Following \cite[Definition 1.1]{kobayashi-thesis} and \cite[Definition 3.1]{bdkim-supersingular-selmer}, we define the \textbf{$\pm$-Selmer groups of $E$ over $\mathbb{Q}_n$} by
\begin{align*}
& \mathrm{Sel}^{\pm}(\mathbb{Q}_n, E[p^\infty]) \\
: = & \mathrm{ker} \left( 
\mathrm{Sel}(\mathbb{Q}_n, E[p^\infty]) \to 
\dfrac{\mathrm{H}^1(\mathbb{Q}_{n,p}, E[p^\infty])}{E^{\pm}(\mathbb{Q}_{n,p}) \otimes \mathbb{Q}_p/\mathbb{Z}_p}
\right)  \\
 = &\mathrm{ker} \left( 
\mathrm{H}^1(\mathbb{Q}_{\Sigma}/\mathbb{Q}_n, E[p^\infty]) \to 
\dfrac{\mathrm{H}^1(\mathbb{Q}_{n,p}, E[p^\infty])}{E^{\pm}(\mathbb{Q}_{n,p}) \otimes \mathbb{Q}_p/\mathbb{Z}_p} \times 
\prod_{w \vert \ell, \ell \in \Sigma, \ell \neq p} \dfrac{\mathrm{H}^1(\mathbb{Q}_{n,w}, E[p^\infty])}{E(\mathbb{Q}_{n,w}) \otimes \mathbb{Q}_p/\mathbb{Z}_p} 
\right)
\end{align*}
and the \textbf{$\pm$-Selmer groups of $E$ over $\mathbb{Q}_\infty$} by
$$\mathrm{Sel}^{\pm}(\mathbb{Q}_\infty, E[p^\infty]) : = \varinjlim_n \mathrm{Sel}^{\pm}(\mathbb{Q}_n, E[p^\infty]) ,$$
respectively.
Note that $\pm$-Selmer groups are also independent of the choice of $\Sigma$ as usual Selmer groups are since the local conditions at the places above $p$ are only changed.
Also, it is easy to see 
$$\mathrm{Sel}_0(\mathbb{Q}_n, E[p^\infty]) \subseteq \mathrm{Sel}^{\pm}(\mathbb{Q}_n, E[p^\infty]) \subseteq \mathrm{Sel}(\mathbb{Q}_n, E[p^\infty]),$$
respectively.

\subsubsection{$\pm$-$p$-adic $L$-functions and $\pm$-Coleman maps} \label{subsubsec:pm-p-adic-L-Coleman}
We recall the characterization of \textbf{$\pm$-$p$-adic $L$-functions} $L^{\pm}_p(\mathbb{Q}_\infty,f) \in \Lambda$ by their interpolation property \cite[(10), (11), and (12)]{pollack-rubin}:
\begin{align*}
\chi \left( L^+_p(\mathbb{Q}_\infty,f) \right) & = (-1)^{(n+1)/2} \cdot \dfrac{\tau(\chi)}{\chi(\widetilde{\omega}^+_n)} \cdot \dfrac{L(E, \chi^{-1}, 1)}{\Omega_E} & \textrm{ if $\chi$ has order $p^n$ with $n$ odd } \\
\chi \left( L^-_p(\mathbb{Q}_\infty,f) \right) & = (-1)^{(n/2) + 1} \cdot \dfrac{\tau(\chi)}{\chi(\widetilde{\omega}^-_n)} \cdot \dfrac{L(E, \chi^{-1}, 1)}{\Omega_E} & \textrm{ if $\chi$ has order $p^n>1$ with $n$ even } \\
\mathbf{1} \left( L^+_p(\mathbb{Q}_\infty,f) \right) & = (p-1) \cdot \dfrac{L(E, 1)}{\Omega_E} \\
\mathbf{1} \left( L^-_p(\mathbb{Q}_\infty,f) \right) & = 2 \cdot \dfrac{L(E, 1)}{\Omega_E}
\end{align*}
where $\chi$ is a character of $\Gamma$, $\mathbf{1}$ is the trivial character, and $\tau(\chi)$ is the Gauss sum of $\chi$.

We also recall \textbf{$\pm$-Coleman maps}. Our sign convention follows that of \cite{kurihara-pollack}.
\begin{thm}[{\cite[Theorem 6.2, Theorem 6.3, and $\S$8]{kobayashi-thesis}}, {\cite[$\S$1.1]{kurihara-pollack}}] \label{thm:pm-coleman}
There exist maps 
$$\mathrm{Col}^{\pm}_n: \mathrm{H}^1(\mathbb{Q}_{n,p}, T) \to  \Lambda_n / \omega^{\mp}_n$$
such that
\begin{enumerate}
\item $\mathrm{Col}^{\pm}_n: \mathrm{H}^1(\mathbb{Q}_{n,p}, T) / \mathrm{ker} \ \mathrm{Col}^{\pm}_n \simeq \Lambda_n / \omega^{\mp}_n$ and
\item $\mathrm{Col}^{\pm}_n(\mathrm{loc} \ \mathbf{z}_{\mathrm{Kato},n}) =  L^{\pm}_p(\mathbb{Q}_n,f)$
\end{enumerate}
where 
$\mathbf{z}_{\mathrm{Kato},n} \in \mathrm{H}^1(\mathbb{Q}_n, T)$ is Kato's zeta element at $\mathbb{Q}_n$ and 
$L^{\pm}_p(\mathbb{Q}_n,f) := L^{\pm}_p(\mathbb{Q}_\infty,f) \Mod{\omega^{\mp}_n} $.
By taking the inverse limit with respect to $n$, we have maps
$$\mathrm{Col}^{\pm}: \mathbb{H}^1_{\mathrm{loc}}(T)  \to  \Lambda$$
such that
\begin{enumerate}
\item $\mathrm{Col}^{\pm}$ are surjective and
\item $\mathrm{Col}^{\pm}(\mathrm{loc} \ \mathbf{z}_{\mathrm{Kato}}) =  L^{\pm}_p(\mathbb{Q}_\infty,f)$
\end{enumerate}
where 
$$\mathbb{H}^1_{\mathrm{loc}}(T) := \varprojlim_n \mathrm{H}^1(\mathbb{Q}_{n,p}, T)$$
is the local Iwasawa cohomology group and
$\mathbf{z}_{\mathrm{Kato}} \in \mathrm{H}^1(\mathbb{Q}_\infty, T)$ is Kato's zeta element at $\mathbb{Q}_\infty$.
\end{thm}
\begin{rem}
The construction of $\mathrm{Col}^{\pm}_n$ in \cite{kobayashi-thesis} uses certain local points of formal groups of elliptic curves via Honda theory, and that in \cite{kurihara-pollack} uses the $P_n$-paring defined by the second named author in \cite{kurihara-invent} and Proposition \ref{prop:pollack_theta_pm}.
\end{rem}
\subsection{$\pm$-main conjectures}
We recall the $\Lambda$-cotorsion property of $\pm$-Selmer groups (\cite[Theorem 7.3.ii)]{kobayashi-thesis}, \cite[Theorem 6.3]{pollack-rubin}).
\begin{thm} \label{thm:pm-torsion}
Let $p$ be an odd prime and $E$ an elliptic curve over $\mathbb{Q}$ with $a_p(E) = 0$.
Then both $\mathrm{Sel}^{+}(\mathbb{Q}_\infty, E[p^\infty])$ and $\mathrm{Sel}^{-}(\mathbb{Q}_\infty, E[p^\infty])$ are $\Lambda$-cotorsion.
\end{thm}
The following statement is the pair of the Iwasawa main conjectures for elliptic curves with supersingular reduction (\cite[Even,odd main conjectures, $\S$4]{kobayashi-thesis}).
\begin{conj}[$\pm$-main conjectures]
\label{conj:pm-main-conj}
Let $p$ be an odd prime and $E$ an elliptic curve over $\mathbb{Q}$ with $a_p(E) = 0$.
Then 
$$ \left( L^{\mp}_p(\mathbb{Q}_\infty, f) \right) = \mathrm{char}_\Lambda \left( \mathrm{Sel}^{\pm}(\mathbb{Q}_\infty, E[p^\infty])^\vee \right) .$$
\end{conj}
As in the ordinary case, the Euler system argument yields the following statement (\cite[Theorem 4.1]{kobayashi-thesis}, \cite[Theorem in Introduction]{pollack-rubin}).
\begin{thm}
\label{thm:pm-main-conj}
Let $p$ be an odd prime and $E$ an elliptic curve over $\mathbb{Q}$ with $a_p(E) = 0$.
\begin{enumerate}
\item If $E$ has CM, then Conjecture \ref{conj:pm-main-conj} holds.
\item If $E$ has no CM, then we assume $\overline{\rho}$ is surjective.
Then 
$$ \left( L^{\mp}_p(\mathbb{Q}_\infty, f) \right)  \subseteq \mathrm{char}_\Lambda \left( \mathrm{Sel}^{\pm}(\mathbb{Q}_\infty, E[p^\infty])^\vee \right) .$$
\end{enumerate}
\end{thm}
\begin{rem} \label{rem:pm-main-conj}
Even in the non-CM case, there are several approaches to establish Conjecture \ref{conj:pm-main-conj} under certain tame level assumptions. Since none of them is published yet, we just record the assumptions they made. 
More precisely, it is announced that the $\pm$-main conjectures hold if the conditions in Theorem \ref{thm:pm-main-conj} and one of the following tame level conditions hold:
\begin{itemize}
\item[\cite{wan-main-conj-nonord}] there exists a prime $q \Vert N$ such that $\overline{\rho}$ is ramifed at $q$,
\item[\cite{wan-main-conj-nonord}] $N$ is square-free and there exist two primes $q \Vert N$ such that $\overline{\rho}$ is ramified at $q$, or
\item[\cite{wan-main-conj-ss-ec}] $N$ is square-free (only assuming the absolute irreducibility of $\overline{\rho}$).
\end{itemize}
In addition, the numerical criterion of K-Kim-Sun described in Theorem \ref{thm:main-conj-2} still works to verify the $\pm$-main conjectures (without these tame level assumptions) in the exactly same way.
If the validity of the results in these preprints is confirmed, then the $\pm$-main conjecture assumption in Theorem \ref{thm:main_theorem_2} could be removed.
Note that any of these results is used in this article.
Especially, Theorem \ref{thm:pm-main-conj}.(2) is strong enough to prove the Mazur-Tate conjecture (Theorem \ref{thm:main_theorem}).
\end{rem}

\subsection{Nonexistence of proper $\Lambda$-submodules of finite index}
We recall B.D. Kim's result \cite[Theorem 1.1]{bdkim-supersingular-selmer} on the analogue of Theorem \ref{thm:no-finite-ordinary} for the supersingular setting. For the further developments along this direction, see \cite{kitajima-otsuki}.
\begin{thm} \label{thm:no-finite}
The Selmer groups $\mathrm{Sel}^{+}(\mathbb{Q}_\infty, E[p^\infty])$ and $\mathrm{Sel}^{-}(\mathbb{Q}_\infty, E[p^\infty])$ have no proper $\Lambda$-submodule of finite index.
\end{thm}

\subsection{$\pm$-exact control theorems}
We recall the $\pm$-version of the control theorem (\cite[Theorem 9.3]{kobayashi-thesis}, \cite[Theorem 6.8]{iovita-pollack}).
\begin{thm}[$\pm$-control theorems] \label{thm:pm-control}
The restriction map
$$\mathrm{Sel}^{\pm}(\mathbb{Q}_n, E[p^\infty])[\omega^{\pm}_n] \to \mathrm{Sel}^{\pm}(\mathbb{Q}_\infty, E[p^\infty])[\omega^{\pm}_n] $$
is injective with the finite cokernel whose size is bounded independently of $n$. If we further assume that $p \nmid \mathrm{Tam}(E)$, then the restriction map
is an isomorphism.
\end{thm}
\begin{proof}
The $a_p(E)=0$ condition ensures that $E(\mathbb{Q})[p]$ is trivial, and it implies that the restriction map is injective as in \cite[Lemma 9.1]{kobayashi-thesis}. The failure of the surjectivity comes only from prime-to-$p$ local conditions; thus, the situation coincides with the ordinary case. The $p \nmid \mathrm{Tam}(E)$ condition ensures that the failure vanishes. See \cite[Theorem 9.3]{kobayashi-thesis} and \cite[Proposition 3.8]{greenberg-lnm} for detail.
\end{proof}

\subsection{The consequence}
\begin{cor} \label{cor:the-key-lemma}
Let $p$ be an odd prime and $E$ an elliptic curve over $\mathbb{Q}$ with  $a_p(E) = 0$.
Assume $\overline{\rho}$ is surjective if $E$ has no CM.
Then we have
$$ \left( \widetilde{\omega}^{\mp}_n \cdot L^{\mp}_p(\mathbb{Q}_\infty,f) \Mod{\omega_n} \right)
\subseteq
\widetilde{\omega}^{\mp}_n \cdot \mathrm{Fitt}_{\Lambda_n} \left(  \mathrm{Sel}^{\pm}(\mathbb{Q}_n, E[p^\infty])^\vee \right) $$
 in $\Lambda_n$, respectively.
\end{cor}
\begin{proof}
By Theorem \ref{thm:pm-torsion}, Theorem \ref{thm:pm-main-conj}, Theorem \ref{thm:no-finite}, and Lemma \ref{lem:char-to-fitt}, we have
$$ \left( L^{\mp}_p(\mathbb{Q}_\infty, f) \right)  \subseteq \mathrm{Fitt}_\Lambda \left( \mathrm{Sel}^{\pm}(\mathbb{Q}_\infty, E[p^\infty])^\vee \right) $$
under the conditions of Theorem \ref{thm:pm-main-conj}.
Taking the quotient by $\omega^{\pm}_n$, we obtain
$$
\left( L^{\mp}_p(\mathbb{Q}_\infty, f) \Mod{\omega^{\pm}_n} \right) \subseteq \mathrm{Fitt}_{\Lambda_n / \omega^{\pm}_n} \left( \left( \mathrm{Sel}^{\pm}(\mathbb{Q}_\infty, E[p^\infty])[\omega^{\pm}_n] \right)^\vee \right)
$$
in $\Lambda_n/\omega^{\pm}_n$, respectively.
By Theorem \ref{thm:pm-control} and Lemma \ref{lem:fitting_quotient}, we obtain
$$\left( L^{\mp}_p(\mathbb{Q}_\infty,f) \Mod{\omega^{\pm}_n} \right) \subseteq \mathrm{Fitt}_{\Lambda_n / \omega^{\pm}_n} \left( \left( \mathrm{Sel}^{\pm}(\mathbb{Q}_n, E[p^\infty])[\omega^{\pm}_n] \right)^\vee \right)$$
in $\Lambda_n/\omega^{\pm}_n$, respectively.
Since we have the following equality
\begin{align*}
\mathrm{Fitt}_{\Lambda_n / \omega^{\pm}_n} \left( \left( \mathrm{Sel}^{\pm}(\mathbb{Q}_n, E[p^\infty])[\omega^{\pm}_n] \right)^\vee \right) & = \dfrac{ \mathrm{Fitt}_{\Lambda_n} \left(  \mathrm{Sel}^{\pm}(\mathbb{Q}_n, E[p^\infty])^\vee \right) + (\omega^{\pm}_n) }{ (\omega^{\pm}_n) }
\end{align*}
in $\Lambda_n / \omega^{\pm}_n$ by Lemma \ref{lem:fitting-ideals-base-change}, we have inclusions
$$\left( L^{\mp}_p(\mathbb{Q}_\infty,f) \Mod{\omega_n} \right)  + (\omega^{\pm}_n) \subseteq \mathrm{Fitt}_{\Lambda_n} \left(  \mathrm{Sel}^{\pm}(\mathbb{Q}_n, E[p^\infty])^\vee \right) + (\omega^{\pm}_n)$$
in $\Lambda_n$, respectively.
Multiplying $\widetilde{\omega}^{\mp}_n$, the conclusion immediately follows.
\end{proof}
\begin{rem}
If we further assume $p \nmid \mathrm{Tam}(E)$ and the $\pm$-main conjectures, then the inclusion in Corollary \ref{cor:the-key-lemma} becomes an equality.
\end{rem}

\section{Comparison of local conditions at $p$} \label{sec:putting}
Consider the exact sequence of $\Lambda_n$-modules (cf. \cite[(4.2)]{kitajima-otsuki})
\[
\xymatrix{
\left( \dfrac{E(\mathbb{Q}_{n,p}) \otimes \mathbb{Q}_p/\mathbb{Z}_p}{E^{\pm}(\mathbb{Q}_{n,p})  \otimes \mathbb{Q}_p/\mathbb{Z}_p} \right)^\vee \ar[r]^-{\iota^{\pm}} & \mathrm{Sel}(\mathbb{Q}_n, E[p^\infty])^\vee  \ar[r] & \mathrm{Sel}^\pm(\mathbb{Q}_n, E[p^\infty])^\vee  \ar[r] & 0 .
}
\]
Then we have
$$\mathrm{Fitt}_{\Lambda_n} \left( \left( \dfrac{E(\mathbb{Q}_{n,p}) \otimes \mathbb{Q}_p/\mathbb{Z}_p}{E^{\pm}(\mathbb{Q}_{n,p})  \otimes \mathbb{Q}_p/\mathbb{Z}_p} \right)^\vee / \mathrm{ker} (\iota^{\pm}) \right) \cdot \mathrm{Fitt}_{\Lambda_n} \left( \mathrm{Sel}^\pm(\mathbb{Q}_n, E[p^\infty])^\vee \right) 
 \subseteq
\mathrm{Fitt}_{\Lambda_n} \left( \mathrm{Sel}(\mathbb{Q}_n, E[p^\infty])^\vee \right) $$
by Lemma \ref{lem:fitting_exact}.
By Lemma \ref{lem:fitting_quotient}, we also have
$$
\mathrm{Fitt}_{\Lambda_n} \left( \left( \dfrac{E(\mathbb{Q}_{n,p}) \otimes \mathbb{Q}_p/\mathbb{Z}_p}{E^{\pm}(\mathbb{Q}_{n,p})  \otimes \mathbb{Q}_p/\mathbb{Z}_p} \right)^\vee \right) \subseteq
\mathrm{Fitt}_{\Lambda_n} \left( \left( \dfrac{E(\mathbb{Q}_{n,p}) \otimes \mathbb{Q}_p/\mathbb{Z}_p}{E^{\pm}(\mathbb{Q}_{n,p})  \otimes \mathbb{Q}_p/\mathbb{Z}_p} \right)^\vee / \mathrm{ker} (\iota^{\pm}) \right) .
$$
We observe that
\begin{align*}
\left( \dfrac{ E(\mathbb{Q}_{n,p}) \otimes \mathbb{Q}_p/\mathbb{Z}_p }{E^{\pm}(\mathbb{Q}_{n,p}) \otimes \mathbb{Q}_p/\mathbb{Z}_p} \right)^\vee & \simeq \left( \dfrac{ \widehat{E}(\mathfrak{m}_{n}) \otimes \mathbb{Q}_p/\mathbb{Z}_p }{\widehat{E}^{\pm}(\mathfrak{m}_{n}) \otimes \mathbb{Q}_p/\mathbb{Z}_p} \right)^\vee  & \textrm{ \cite[Lemma 3.14]{kitajima-otsuki} }\\
& \simeq \left( \dfrac{ \widehat{E}(\mathfrak{m}_{n})  }{\widehat{E}^{\pm}(\mathfrak{m}_{n}) } \otimes \mathbb{Q}_p/\mathbb{Z}_p \right)^\vee .
\end{align*}
Due to \cite[Proposition 4.11]{iovita-pollack}, we have the following exact sequence
\[
\xymatrix{
0 \ar[r] & \widehat{E}(p\mathbb{Z}_p) \ar[r]^-{f} \ar[d]_-{\simeq} & \widehat{E}^+(\mathfrak{m}_n) \oplus \widehat{E}^-(\mathfrak{m}_n) \ar[r]^-{g} \ar[d]_-{\simeq} & \widehat{E}(\mathfrak{m}_n) \ar[r] \ar[d]_-{\simeq}^-{\textrm{\cite[Proposition 5.8]{iovita-pollack}}} & 0  \\
0 \ar[r] &  \widetilde{\omega}^{+}_n \widetilde{\omega}^{-}_n  \Lambda_n \ar[r] & \widetilde{\omega}^-_n\Lambda_n \oplus \widetilde{\omega}^+_n\Lambda_n \ar[r] & \left( \widetilde{\omega}^{+}_n, \widetilde{\omega}^{-}_n \right) \Lambda_n \ar[r] & 0
}
\]
where $f$ is the diagonal embedding and $g: (a,b) \mapsto a-b$.
Note that $\widetilde{\omega}^{+}_n \widetilde{\omega}^{-}_n  \Lambda_n \simeq \Lambda_n / X\Lambda_n  \simeq \mathbb{Z}_p  $.
This implies that
$$\widehat{E}(\mathfrak{m}_n) / \widehat{E}^{\pm}(\mathfrak{m}_n) \simeq  \left( \widetilde{\omega}^{+}_n, \widetilde{\omega}^{-}_n \right) \Lambda_n / \widetilde{\omega}^{\mp}_n\Lambda_n .$$
Then
$$\mathrm{Fitt}_{\Lambda_n} \left( \left( \frac{ \left( \widetilde{\omega}^{+}_n, \widetilde{\omega}^{-}_n \right) \Lambda_n }{ \widetilde{\omega}^{\mp}_n\Lambda_n } \otimes \mathbb{Q}_p/\mathbb{Z}_p \right)^\vee  \right)
=
\mathrm{Fitt}_{\Lambda_n} \left( \left( \left( \widehat{E}(\mathfrak{m}_n) / \widehat{E}^{\pm}(\mathfrak{m}_n) \right) \otimes \mathbb{Q}_p/\mathbb{Z}_p \right)^\vee  \right)  .$$

The following proposition is due to Robert Pollack.
\begin{prop}
$$\mathrm{Fitt}_{\Lambda_n} \left( \left( \frac{ \left( \widetilde{\omega}^{+}_n, \widetilde{\omega}^{-}_n \right) \Lambda_n }{ \widetilde{\omega}^{\mp}_n\Lambda_n } \otimes \mathbb{Q}_p/\mathbb{Z}_p \right)^\vee  \right)= \widetilde{\omega}^{\mp}_n\Lambda_n,$$ respectively.
\end{prop}
\begin{proof}
Since the multiplication by $\omega^{\pm}_n$ induces an isomorphism $\Lambda_n / \widetilde{\omega}^{\mp}_n\Lambda_n \simeq \omega^{\pm}_n \Lambda_n$,
we have
$$  \Lambda_n / \widetilde{\omega}^{\mp}_n \simeq \frac{ \left( \widetilde{\omega}^{+}_n, \widetilde{\omega}^{-}_n \right) \Lambda_n }{ \widetilde{\omega}^{\mp}_n\Lambda_n } .$$
We compute
\begin{align*}
\left( \frac{ \left( \widetilde{\omega}^{+}_n, \widetilde{\omega}^{-}_n \right) \Lambda_n }{ \widetilde{\omega}^{\mp}_n\Lambda_n } \otimes \mathbb{Q}_p/\mathbb{Z}_p \right)^\vee & \simeq
 \mathrm{Hom}_{\mathbb{Z}_p}\left( \frac{ \left( \widetilde{\omega}^{+}_n, \widetilde{\omega}^{-}_n \right) \Lambda_n }{ \widetilde{\omega}^{\mp}_n\Lambda_n } \otimes \mathbb{Q}_p/\mathbb{Z}_p, \mathbb{Q}_p/\mathbb{Z}_p \right) \\
& \simeq
 \mathrm{Hom}_{\mathbb{Z}_p}\left( \Lambda_n / \widetilde{\omega}^{\mp}_n , \mathbb{Z}_p \right).
\end{align*}
A direct calculation shows the following identities:
\begin{align*}
 \widetilde{\omega}^{+, \iota}_n & := \prod_{2 \leq m \leq n, m: \textrm{ even}}\Phi_m \left(\frac{1}{1+X} \right) \\ 
 & = \prod_{2 \leq m \leq n, m: \textrm{ even}} \left( \Phi_m (1+X) \cdot (1+X)^{-p^{m-1}(p-1)} \right) \\  
 & = \left( \prod_{2 \leq m \leq n, m: \textrm{ even}} (1+X)^{-p^{m-1}(p-1)} \right)  \cdot  \widetilde{\omega}^{+}_n ,
\end{align*}
and
\begin{align*}
 \widetilde{\omega}^{-, \iota}_n & := \prod_{1 \leq m \leq n, m: \textrm{ odd}}\Phi_m \left(\frac{1}{1+X} \right) \\ 
 & = \prod_{1 \leq m \leq n, m: \textrm{ odd}} \left( \Phi_m (1+X) \cdot (1+X)^{-p^{m-1}(p-1)} \right) \\  
 & = \left( \prod_{1 \leq m \leq n, m: \textrm{ odd}} (1+X)^{-p^{m-1}(p-1)} \right)  \cdot  \widetilde{\omega}^{-}_n .
\end{align*}
We write 
\[
\xymatrix@R=0em
{
 { \displaystyle
c^+ = \left( \prod_{2 \leq m \leq n, m: \textrm{ even}} (1+X)^{-p^{m-1}(p-1)} \right) } ,&
{ \displaystyle
c^- = \left( \prod_{1 \leq m \leq n, m: \textrm{ odd}} (1+X)^{-p^{m-1}(p-1)} \right) 
} 
}
\]
and note that they are invertible in $\Lambda_n$.

We consider a perfect pairing
$\Lambda_n \times \Lambda_n \to \mathbb{Z}_p$ defined by $(\sigma, \tau) = 1$ if $\tau = \sigma^{-1}$ and
 $(\sigma, \tau) = 0$ otherwise where $\sigma, \tau \in \mathrm{Gal}(\mathbb{Q}_{n}/\mathbb{Q})$. Then the pairing induces an isomorphism
$$\mathrm{Hom}_{\mathbb{Z}_p}(\Lambda_n, \mathbb{Z}_p) \simeq \Lambda_n $$
with the reversed $\Lambda_n$-action.
Then we have
\begin{align*}
\mathrm{Hom}_{\mathbb{Z}_p}(\Lambda_n/\widetilde{\omega}^{\mp}_n, \mathbb{Z}_p)
& \simeq \mathrm{Hom}_{\mathbb{Z}_p}(\Lambda_n, \mathbb{Z}_p)[\widetilde{\omega}^{\mp}_n] \\
& \simeq \Lambda_n[\widetilde{\omega}^{\mp, \iota}_n] \\
& \simeq \Lambda_n[\widetilde{\omega}^{\mp}_n] & (c^{\mp}_n \in \Lambda^\times_n)\\
& \simeq  \omega^{\pm}_n \Lambda_n .
\end{align*}
\end{proof}

To sum up, we have
\begin{align*}
\widetilde{\omega}^{\mp}_n \cdot \mathrm{Fitt}_{\Lambda_n} \left( \mathrm{Sel}^\pm(\mathbb{Q}_n, E[p^\infty])^\vee \right)
& =
\mathrm{Fitt}_{\Lambda_n} \left( \left( \frac{ \left( \widetilde{\omega}^{+}_n, \widetilde{\omega}^{-}_n \right) \Lambda_n }{ \widetilde{\omega}^{\mp}_n\Lambda_n } \otimes \mathbb{Q}_p/\mathbb{Z}_p \right)^\vee  \right) \cdot \mathrm{Fitt}_{\Lambda_n} \left( \mathrm{Sel}^\pm(\mathbb{Q}_n, E[p^\infty])^\vee \right) \\
& =
\mathrm{Fitt}_{\Lambda_n} \left( \left( \widehat{E}(\mathfrak{m}_n) / \widehat{E}^{\pm}(\mathfrak{m}_n) \otimes \mathbb{Q}_p/\mathbb{Z}_p \right)^\vee  \right) \cdot \mathrm{Fitt}_{\Lambda_n} \left( \mathrm{Sel}^\pm(\mathbb{Q}_n, E[p^\infty])^\vee \right) \\
& \subseteq
\mathrm{Fitt}_{\Lambda_n} \left( \left( \dfrac{E(\mathbb{Q}_{n,p}) \otimes \mathbb{Q}_p/\mathbb{Z}_p}{E^{\pm}(\mathbb{Q}_{n,p})  \otimes \mathbb{Q}_p/\mathbb{Z}_p} \right)^\vee / \mathrm{ker} (\iota^{\pm}) \right) \cdot \mathrm{Fitt}_{\Lambda_n} \left( \mathrm{Sel}^\pm(\mathbb{Q}_n, E[p^\infty])^\vee \right) \\
& \subseteq
\mathrm{Fitt}_{\Lambda_n} \left( \mathrm{Sel}(\mathbb{Q}_n, E[p^\infty])^\vee \right) .
\end{align*}
By Corollary \ref{cor:the-key-lemma}, we have
$$\left( \widetilde{\omega}^{\mp}_n \cdot L^{\mp}_p(\mathbb{Q}_\infty,f) \Mod{\omega_n} \right) \subseteq \mathrm{Fitt}_{\Lambda_n} \left( \mathrm{Sel}(\mathbb{Q}_n, E[p^\infty])^\vee \right) .$$
Then Theorem \ref{thm:main_theorem} immediately follows from Proposition \ref{prop:pollack_theta_pm}.
Notably, the weak main conjecture of Mazur-Tate holds.

\section{Towards the strong main conjecture} \label{sec:non-CM}
The goal of this section is to prove the inclusion
$$\mathrm{Fitt}_{\Lambda_n}\left( \mathrm{Err}_n \right) \cdot \mathrm{Fitt}_{\Lambda_n}\left( \mathrm{Sel}(\mathbb{Q}_n, E[p^\infty])^\vee \right) \subseteq \left( \theta_n(f), \nu_{n-1,n} \left(  \theta_{n-1}(f) \right) \right)$$
in Theorem \ref{thm:main_theorem_2}. Note that the inclusion gives us a \emph{lower} bound of Selmer groups (up to some error).
Throughout this section, we assume
\begin{enumerate}
\item $\overline{\rho}$ is surjective ($\Rightarrow$ $E$ is automatically non-CM),
\item $p$ does not divide $\mathrm{Tam}(E)$, and
\item the $\pm$-main conjectures (Conjecture \ref{conj:pm-main-conj}).
\end{enumerate}

\subsection{Kato's main conjecture and fine Selmer groups} \label{subsec:kato_main_conj_n_fine_selmer}
Let $j : \mathrm{Spec}(\mathbb{Q}_n) \to \mathrm{Spec}(\mathcal{O}_{\mathbb{Q}_n}[1/p])$ be the natural map.
Let 
\[
\xymatrix{
\mathbb{H}^i_{\mathrm{glob}}(T) := \varprojlim_{n} \mathrm{H}^i_{\mathrm{\acute{e}t}}( \mathrm{Spec}(\mathcal{O}_{\mathbb{Q}_n}[1/p]), j_*T) , & \mathbb{H}^i_{\mathrm{glob}}(V) := \mathbb{H}^i_{\mathrm{glob}}(T) \otimes \mathbb{Q}_p
}
\]
where $\mathrm{H}^i_{\mathrm{\acute{e}t}}( \mathrm{Spec}(\mathcal{O}_{\mathbb{Q}_n}[1/p]), j_*T)$ is the \'{e}tale cohomology group. 
It is well known that $\mathbb{H}^1_{\mathrm{glob}}(T) \simeq \varprojlim_n \mathrm{H}^1(\mathbb{Q}_\Sigma/\mathbb{Q}_n, T)$.
See \cite[$\S$6]{kurihara-invent} and \cite[Proposition 7.1.(i)]{kobayashi-thesis} for detail.

The following theorem is due to Kato (\cite[Theorem 12.4.(1) and (3)]{kato-euler-systems}).
\begin{thm} \label{thm:kato_iwasawa_cohomologies}
Assume that $\overline{\rho}$ is surjective. Then:
\begin{enumerate}
\item $\mathbb{H}^2_{\mathrm{glob}}(T)$ is a finitely generated torsion module over $\Lambda$.
\item $\mathbb{H}^1_{\mathrm{glob}}(T)$ is free of rank one over $\Lambda$.
\end{enumerate}
\end{thm}
We recall the Iwasawa main conjecture without $p$-adic zeta functions \`{a} la Kato and Perrin-Riou (\cite[Conjecture 12.10]{kato-euler-systems}).
\begin{conj}[Kato's main conjecture] \label{conj:kato-main-conjecture}
$$\mathrm{char}_{\Lambda} \left( \left( \mathbb{H}^1_{\mathrm{glob}}(T) / \Lambda \mathbf{z}_{\mathrm{Kato}} \right)_{\mathrm{tors}} \right)
= \mathrm{char}_{\Lambda} \left( \mathbb{H}^2_{\mathrm{glob}}(T) \right)$$
where $M_{\mathrm{tors}}$ is the $\Lambda$-torsion submodule of $M$.
\end{conj}
Note that we crucially use Conjecture \ref{conj:kato-main-conjecture} in the argument.
\begin{rem} $ $
\begin{enumerate}
\item 
If $a_p(E) = 0$, then Kato's main conjecture (Conjecture \ref{conj:kato-main-conjecture}) and the $\pm$-main conjectures (Conjecture \ref{conj:pm-main-conj}) are equivalent due to \cite[Theorem 7.4]{kobayashi-thesis}.
\item
Also, $\mathbb{H}^2_{\mathrm{glob}}(T)$ and $\mathrm{Sel}_0(\mathbb{Q}_\infty, E[p^\infty])^\vee$ are pseudo-isomorphic as $\Lambda$-modules. See  \cite[$\S$6]{kurihara-invent}, \cite[Theorem 7.1.ii)]{kobayashi-thesis} for detail.
\end{enumerate}
\end{rem}
\subsection{Selmer groups and fine Selmer groups in finite layers} \label{subsec:selmer_fine_selmer_finite_layers}
Let 
$$\mathcal{Y}'_n := \mathrm{coker} \left( \mathbb{H}^1_{\mathrm{glob}}(T)_{\Gamma_n} \to \frac{ \mathrm{H}^1(\mathbb{Q}_{n,p}, T) }{ E( \mathbb{Q}_{n,p} ) \otimes \mathbb{Z}_p }  \right),$$
$$\mathcal{Y}_n := \mathrm{coker} \left( \mathrm{H}^1(\mathbb{Q}_{\Sigma}/\mathbb{Q}_{n}, T) \to \frac{ \mathrm{H}^1(\mathbb{Q}_{n,p}, T) }{ E( \mathbb{Q}_{n,p} ) \otimes \mathbb{Z}_p }  \right),$$
and
$$\mathcal{Z}_n := \mathrm{im} \left( \mathrm{H}^1(\mathbb{Q}_{\Sigma}/\mathbb{Q}_{n}, T) \to \frac{ \mathrm{H}^1(\mathbb{Q}_{n,p}, T) }{ E( \mathbb{Q}_{n,p} ) \otimes \mathbb{Z}_p }  \right).$$
Consider the following commutative diagram
\begin{equation}\label{eqn:coker_g_n}
\begin{gathered}
\xymatrix{
& \mathrm{ker} \ g_n \ar[r] \ar[d] & 0 \ar[r] \ar[d] &  \mathrm{ker} \ f_n  \ar[d] \\
& \left( \mathbb{H}^1_{\mathrm{glob}}(T) \right)_{\Gamma_n} \ar[r] \ar[d]_-{g_n} & \frac{ \mathrm{H}^1(\mathbb{Q}_{n,p}, T) }{ E( \mathbb{Q}_{n,p} ) \otimes \mathbb{Z}_p } \ar[r] \ar[d]^-{\simeq} & \mathcal{Y}'_n \ar[r] \ar[d]^-{f_n} & 0 \\
0 \ar[r] & \mathcal{Z}_n \ar[r] \ar[d] & \frac{ \mathrm{H}^1(\mathbb{Q}_{n,p}, T) }{ E( \mathbb{Q}_{n,p} ) \otimes \mathbb{Z}_p } \ar[r] \ar[d] & \mathcal{Y}_n \ar[r] \ar[d] & 0 \\
&   \mathrm{coker} \ g_n \ar[r] & 0 \ar[r] & 0
}
\end{gathered}
\end{equation}
with $$\mathrm{ker} \ f_n \simeq \mathrm{coker} \ g_n$$ by snake lemma.
\begin{rem} \label{rem:coker_g_n}
Here, $\mathrm{coker} \ g_n$ is exactly $\mathrm{Err}_n$. We use the notation $\mathrm{coker} \ g_n$ in this and the next sections.
\end{rem}
Let $\mathcal{Y}'_n / \mathrm{coker} \ g_n := \mathcal{Y}'_n / \mathrm{ker} \ f_n \subseteq \mathcal{Y}_n$.
Then we have
$$\mathrm{Fitt}_{\Lambda_n} \left( \mathcal{Y}_n \right)  \subseteq
\mathrm{Fitt}_{\Lambda_n} \left( \mathcal{Y}'_n / \mathrm{coker} \ g_n \right) $$
by Lemma \ref{lem:fitting_sub}. Using the Poitou-Tate sequence (\cite[A.3.2.Proposition]{perrin-riou-book}, \cite[(7.18)]{kobayashi-thesis}), we have the following exact sequence with splitting
\begin{equation} \label{eqn:splitting_sequence}
\xymatrix@R=0.3em{
\mathrm{H}^1(\mathbb{Q}_{\Sigma}/\mathbb{Q}_n, T ) \ar[r] & \frac{\mathrm{H}^1(\mathbb{Q}_{n,p}, T )}{E(\mathbb{Q}_{n,p}) \otimes \mathbb{Z}_p} \ar[rr] \ar[rd] &  & \mathrm{Sel}(\mathbb{Q}_n, E[p^\infty])^\vee  \ar[r] & \mathrm{Sel}_0(\mathbb{Q}_n, E[p^\infty])^\vee 
\ar[r] & 0 \\
& & \mathcal{Y}_n \ar[ru] \ar[rd] \\
& 0 \ar[ru] & & 0 .
}
\end{equation}

\subsection{A presentation of the difference between Selmer groups and fine Selmer groups}
It would be desirable to compute a presentation matrix of $\mathcal{Y}_n$ from the following exact sequence
\[\xymatrix{
\mathrm{H}^1(\mathbb{Q}_{\Sigma}/\mathbb{Q}_n, T ) \ar[r] & \frac{\mathrm{H}^1(\mathbb{Q}_{n,p}, T )}{E(\mathbb{Q}_{n,p}) \otimes \mathbb{Z}_p} \ar[r] &  \mathcal{Y}_n \ar[r] & 0 .
}
\]
Unfortunately, it seems out of reach with current techniques; instead, we compute a slightly easier version, a presentation matrix of $\mathcal{Y}'_n$ from the following exact sequence
\[\xymatrix{
\mathbb{H}^1_{\mathrm{glob}}(T )_{\Gamma_n} \ar[r] & \frac{\mathrm{H}^1(\mathbb{Q}_{n,p}, T )}{E(\mathbb{Q}_{n,p}) \otimes \mathbb{Z}_p} \ar[r] &  \mathcal{Y}'_n \ar[r] & 0 .
}
\]
We regard $\mathcal{Y}'_n$ as the quotient of $\mathrm{H}^1(\mathbb{Q}_{n,p}, T )$ by local constraint $E(\mathbb{Q}_{n,p}) \otimes \mathbb{Z}_p$ and global constraint $\mathbb{H}^1_{\mathrm{glob}}(T )_{\Gamma_n}$.

\subsubsection{The generators}
Let $\mathbb{H}^1_{\mathrm{loc}}(T)$ be the local Iwasawa cohomology group (defined in Theorem \ref{thm:pm-coleman}).
Since $E$ is supersingular at $p$, $E[p]$ is irreducible as a $\mathrm{Gal}(\overline{\mathbb{Q}}_p/\mathbb{Q}_p)$-module.
Then $\mathbb{H}^1_{\mathrm{loc}}(T)$ is free of rank 2 over $\Lambda$ since $\mathrm{H}^1(\mathbb{Q}_{p}, T)$ is free of rank 2 over $\mathbb{Z}_p$.
\begin{prop}[{\cite[Proposition 1.2]{kurihara-pollack}}] \label{prop:coker_coleman}
Let $\mathrm{Col} := \mathrm{Col}^+ \oplus \mathrm{Col}^-$.
The following sequence 
\[
\xymatrix{
0 \ar[r] & \mathbb{H}^1_{\mathrm{loc}}(T) \ar[rr]^-{\mathrm{Col}} & & \Lambda \oplus \Lambda \ar[r]^-{r} & \mathbb{Z}_p \ar[r] & 0 
}
\]
 is exact where $r (h(X), k(X)) := h(0) - \frac{p-1}{2} \cdot k(0)$.
\end{prop}
We pick a $\Lambda$-basis $(e_1, e_2)$ of $\mathbb{H}^1_{\mathrm{loc}} (T) = \mathrm{ker} (r)$ by
\[
\xymatrix{
\mathrm{Col}(e_1) = (\frac{p-1}{2}, 1), & \mathrm{Col}(e_2)  = (X, 0) .
}
\]
Then 
\begin{align*}
\mathbb{H}^1_{\mathrm{loc}} (T) & = \Lambda e_1  \oplus \Lambda e_2 \\
& \simeq \Lambda \mathrm{Col}(e_1) \oplus \Lambda \mathrm{Col}(e_2)  \\
& \subseteq \Lambda \oplus \Lambda .
\end{align*}
By the irreducibility of $E[p]$ as a $\mathrm{Gal}(\overline{\mathbb{Q}}_p/\mathbb{Q}_p)$-module, we have
\begin{align*}
 \mathbb{H}^1_{\mathrm{loc}} (T)_{\Gamma_n} & = \mathrm{H}^1(\mathbb{Q}_{n,p}, T) \\
& = \Lambda_n e_1 \oplus \Lambda_n e_2 .
\end{align*}
\subsubsection{The local constraint}
Consider the exact sequence
\[
\xymatrix{
 0 \ar[r] & \dfrac{\mathrm{H}^1(\mathbb{Q}_{n,p}, T )}{E(\mathbb{Q}_{n,p}) \otimes \mathbb{Z}_p} \ar[r] \ar@{^{(}->}[rd] & \dfrac{\mathrm{H}^1(\mathbb{Q}_{n,p}, T )}{\mathrm{ker}(\mathrm{Col}^+_n)} \oplus \dfrac{\mathrm{H}^1(\mathbb{Q}_{n,p}, T )}{\mathrm{ker}(\mathrm{Col}^-_n)} \ar[r] \ar[d]^-{\mathrm{Col}_n := \mathrm{Col}^+_n \oplus \mathrm{Col}^-_n}_-{\simeq} & \dfrac{\mathrm{H}^1(\mathbb{Q}_{p}, T )}{E(\mathbb{Q}_{p}) \otimes \mathbb{Z}_p} \ar[r] & 0 \\
& & \Lambda_n/ \omega^-_n \oplus \Lambda_n/ \omega^+_n .
}
\]
We investigate the image of $e_1$ and $e_2$ in $\Lambda_n/ \omega^-_n \oplus \Lambda_n/ \omega^+_n$ under $\mathrm{Col}_n$.
Then we naturally obtain the following relations of $\frac{\mathrm{H}^1(\mathbb{Q}_{n,p},T)}{E(\mathbb{Q}_{n,p}) \otimes \mathbb{Z}_p}$:
\begin{align*}
\widetilde{\omega}^-_n \cdot e_2 & = (\widetilde{\omega}^-_n X, 0) \\
& = ( \omega^-_n, 0) \\
& = (0,0) \in \Lambda_n/ \omega^-_n \oplus \Lambda_n/ \omega^+_n
\end{align*}
\begin{align*}
\omega^+_n \cdot e_1  - \frac{p-1}{2} \cdot \widetilde{\omega}^+_n \cdot e_2 
& = ( \omega^+_n \cdot \frac{p-1}{2} , \omega^+_n) - (\frac{p-1}{2} \cdot \widetilde{\omega}^+_n \cdot X, 0)\\
& = ( 0 , \omega^+_n) \\
& = (0,0) \in \Lambda_n/ \omega^-_n \oplus \Lambda_n/ \omega^+_n .
\end{align*}
Also, since $E(\mathbb{Q}_{n,p}) \otimes \mathbb{Z}_p$ is generated by two elements over $\Lambda_n$ via a formal group argument as in \cite[Proposition 4.11]{iovita-pollack}, 
we know that $\frac{\mathrm{H}^1(\mathbb{Q}_{n,p},T)}{E(\mathbb{Q}_{n,p}) \otimes \mathbb{Z}_p}$ is the module with 
two generators $e_1$, $e_2$ and the above two relations (these relations 
are all).

\subsubsection{The global constraint} \label{subsubsec:global_constraint}
Due to Theorem \ref{thm:kato_iwasawa_cohomologies}.(2), we have
$$\mathbb{H}^1_{\mathrm{glob}} (T) \simeq \Lambda$$
and let $b$ be a $\Lambda$-generator of $\mathbb{H}^1_{\mathrm{glob}} (T)$.
Then $b$ is also a $\Lambda_n$-generator of $\mathbb{H}^1_{\mathrm{glob}}(T)_{\Gamma_n} \simeq \Lambda_n $.
We write the image of $b$ by $(b_1,b_2)$ under the map
\[
\xymatrix@R=0em{
\mathbb{H}^1_{\mathrm{glob}}(T) \ar[r]^-{\mathrm{loc}} & \mathbb{H}^1_{\mathrm{loc}}(T) \ar[rr]^-{\mathrm{Col}^{+} \oplus \mathrm{Col}^-} & & \Lambda\oplus \Lambda \\
b \ar@{|->}[rrr] & & & (b_1,b_2)
}
\]

Since 
$$\frac{\mathrm{H}^1(\mathbb{Q}_{n,p}, T )}{E(\mathbb{Q}_{n,p}) \otimes \mathbb{Z}_p} \hookrightarrow \Lambda_n/ \omega^-_n \oplus \Lambda_n/ \omega^+_n ,$$
we have
$$ \mathcal{Y}'_n = \frac{\mathrm{H}^1(\mathbb{Q}_{n,p}, T )}{E(\mathbb{Q}_{n,p}) \otimes \mathbb{Z}_p + \mathrm{im} \  \mathbb{H}^1_{\mathrm{glob}}(T)} \hookrightarrow \frac{\Lambda_n/ \omega^-_n \oplus \Lambda_n/ \omega^+_n}{(b_1, b_2)}.$$
where $\mathrm{im} \ \mathbb{H}^1_{\mathrm{glob}}(T)$ is the image of $\mathbb{H}^1_{\mathrm{glob}}(T)$
in $\mathrm{H}^1(\mathbb{Q}_{n,p}, T )$ and it is a quotient of $\mathbb{H}^1_{\mathrm{glob}}(T)_{\Gamma_n}$.
Then
\begin{align*}
b_2 e_1 - \frac{b_1 - \frac{p-1}{2} b_2}{X} e_2 & = b_2 (\frac{p-1}{2}, 1) + \frac{b_1 - \frac{p-1}{2} b_2}{X} (X,0) \\
& = (b_1, b_2) \\
& = (0,0) \in \mathcal{Y}'_n 
\end{align*} 

\subsubsection{A presentation matrix}
Using all the above discussion on generators and relations arising from
\[\xymatrix{
\mathbb{H}^1_{\mathrm{glob}}(T )_{\Gamma_n} \ar[r] & \frac{\mathrm{H}^1(\mathbb{Q}_{n,p}, T )}{E(\mathbb{Q}_{n,p}) \otimes \mathbb{Z}_p} \ar[r] &  \mathcal{Y}'_n \ar[r] & 0 ,
}
\]
we know there are 2 generators and 3 relations. Now we describe a presentation matrix $A$ of $\mathcal{Y}'_n$ over $\Lambda_n$
\[
\xymatrix{
\left( \Lambda_n \right)^{\oplus 3} \ar[r]^-{A} & \Lambda \mathrm{Col}(e_1) \oplus \Lambda \mathrm{Col}(e_2) \ar[r] & \mathcal{Y}'_n \ar[r] & 0
}
\]
by
$$A = \left( 
\begin{matrix}
0 & \omega^+_n & b_2\\
\widetilde{\omega}^-_n  & - \frac{p-1}{2}\widetilde{\omega}^+_n  &  \frac{  b_1 - \frac{p-1}{2}b_2}{X}
\end{matrix}
\right) .$$
A direct computation of minors of the above matrix $A$ yields the following 
statement.
\begin{prop} \label{prop:fitting_difference}
$$\mathrm{Fitt}_{\Lambda_n}(\mathcal{Y}'_n) = ( \widetilde{\omega}^+_n b_1,  \widetilde{\omega}^-_n b_2) .$$
\end{prop}

\subsection{Putting it all together}
Consider the exact sequence
\[
\xymatrix{
0 \ar[r] & \left( \mathrm{Sel}_0(\mathbb{Q}_\infty, E[p^\infty])^\vee \right)_{\mathrm{mft}}
\ar[r] &  \mathrm{Sel}_0(\mathbb{Q}_\infty, E[p^\infty])^\vee
\ar[r] & \mathcal{S}
\ar[r] & 0
}
\]
where $M_{\mathrm{mft}}$ is the maximal finite torsion $\Lambda$-submodule of $M$.
Then we have
$$\mathrm{char}_{\Lambda} \left( \left( \mathrm{Sel}_0(\mathbb{Q}_\infty, E[p^\infty])^\vee \right)_{\mathrm{mft}} \right)
\cdot \mathrm{char}_{\Lambda} \left( \mathcal{S} \right) = 
\mathrm{char}_{\Lambda} \left( \mathrm{Sel}_0(\mathbb{Q}_\infty, E[p^\infty])^\vee \right) .$$

Since $\mathrm{char}_{\Lambda} \left( \left( \mathrm{Sel}_0(\mathbb{Q}_\infty, E[p^\infty])^\vee \right)_{\mathrm{mft}} \right)$ is trivial and the projective dimension of $\mathcal{S}$ over $\Lambda$ is $\leq 1$ (i.e. $\mathrm{pd}_\Lambda \mathcal{S} \leq 1$), we have
\begin{align*}
\mathrm{Fitt}_{\Lambda} \left( \mathcal{S} \right) & = \mathrm{char}_{\Lambda} \left( \mathcal{S} \right) \\
& = \mathrm{char}_{\Lambda} \left( \mathrm{Sel}_0(\mathbb{Q}_\infty, E[p^\infty])^\vee \right) \\
& = \mathrm{char}_{\Lambda} \left( \mathbb{H}^1_{\mathrm{glob}}(T) / \Lambda \mathbf{z}_{\mathrm{Kato}}  \right) 
\end{align*}
where Lemma \ref{lem:char-to-fitt} and Conjecture \ref{conj:kato-main-conjecture} are used to obtain the first and the third equalities, respectively. Then, by Kato's main conjecture (Conjecture \ref{conj:kato-main-conjecture}) again, we have
$$\mathrm{Fitt}_{\Lambda} \left( \mathcal{S} \right) = ( c ) \subseteq \Lambda$$
where $\mathbf{z}_{\mathrm{Kato}} = c \cdot b$ in $\mathbb{H}^1_{\mathrm{glob}} (T)$ with $b$ the chosen $\Lambda$-generator of $\mathbb{H}^1_{\mathrm{glob}} (T)$ in $\S$\ref{subsubsec:global_constraint}.

By the control theorem for fine Selmer groups (\cite[Lemma 4.2 and Remark 4.4]{kurihara-invent}), we have
$$\mathrm{Sel}_0(\mathbb{Q}_n, E[p^\infty])^\vee \simeq \left( \mathrm{Sel}_0(\mathbb{Q}_\infty, E[p^\infty])^\vee \right)_{\Gamma_n}.$$
Consider two exact sequences with compatibility
\[
\xymatrix{
0 \ar[r] & \mathcal{Y}_n \ar[r] \ar@{_{(}->}[d] & \mathrm{Sel}(\mathbb{Q}_n, E[p^\infty])^\vee  \ar[r] \ar@{=}[d] & \left( \mathrm{Sel}_0(\mathbb{Q}_\infty, E[p^\infty])^\vee \right)_{\Gamma_n} \ar[r] \ar@{->>}[d] & 0 \\
0 \ar[r] & \mathcal{A}_n \ar[r] & \mathrm{Sel}(\mathbb{Q}_n, E[p^\infty])^\vee \ar[r] &  \mathcal{S}_{\Gamma_n} \ar[r] & 0 
}
\]
where $\mathcal{A}_n$ is defined to be the kernel of the map $\mathrm{Sel}(\mathbb{Q}_n, E[p^\infty])^\vee \to  \mathcal{S}_{\Gamma_n}$.

Since $\mathrm{pd}_{\Lambda} \mathcal{S} \leq 1$, we have $\mathcal{S}$ admits a presentation by a square matrix over $\Lambda$. Thus, $\mathcal{S}_{\Gamma_n}$ also admits a presentation by a square matrix over $\Lambda_n$.
Then we have
\begin{align*}
\mathrm{Fitt}_{\Lambda_n} \left( \mathrm{Sel}(\mathbb{Q}_n, E[p^\infty])^\vee \right) & =
 \mathrm{Fitt}_{\Lambda_n} \left( \mathcal{A}_n \right) \cdot \mathrm{Fitt}_{\Lambda_n} \left( \mathcal{S}_{\Gamma_n} \right)  \\
& \subseteq 
\mathrm{Fitt}_{\Lambda_n} \left( \mathcal{Y}_n \right) \cdot \mathrm{Fitt}_{\Lambda_n} \left( \mathcal{S}_{\Gamma_n} \right) 
\end{align*}
where Lemma \ref{lem:fitting_exact_refined_square} and Lemma \ref{lem:fitting_sub} are used to obtain the first equality and the second inclusion, respectively.
Multiplying $\mathrm{Fitt}_{\Lambda_n} \left( \mathrm{coker} \ g_n \right)$, we have
\begin{align*}
 \mathrm{Fitt}_{\Lambda_n} \left( \mathrm{coker} \ g_n \right) \cdot \mathrm{Fitt}_{\Lambda_n} \left( \mathrm{Sel}(\mathbb{Q}_n, E[p^\infty])^\vee \right) 
& \subseteq \mathrm{Fitt}_{\Lambda_n} \left( \mathrm{coker} \ g_n \right) \cdot \mathrm{Fitt}_{\Lambda_n} \left( \mathcal{Y}_n \right) \cdot \mathrm{Fitt}_{\Lambda_n} \left( \mathcal{S}_{\Gamma_n} \right) \\
&  \subseteq \mathrm{Fitt}_{\Lambda_n} \left( \mathrm{coker} \ g_n \right) \cdot  \mathrm{Fitt}_{\Lambda_n} \left( \mathcal{Y}'_n / \mathrm{coker} \ g_n \right) \cdot \mathrm{Fitt}_{\Lambda_n} \left( \mathcal{S}_{\Gamma_n} \right) \\
&  \subseteq  \mathrm{Fitt}_{\Lambda_n} \left( \mathcal{Y}'_n \right) \cdot \mathrm{Fitt}_{\Lambda_n} \left( \mathcal{S}_{\Gamma_n} \right) \\
&  =  \left( \widetilde{\omega}^+_n b_1,  \widetilde{\omega}^-_n b_2 \right) \cdot \left( c \right)\\
&  =  \left( \widetilde{\omega}^+_n \mathrm{Col}^+( \mathrm{loc} \ b ),  \widetilde{\omega}^-_n \mathrm{Col}^-( \mathrm{loc} \ b ) \right) \cdot \left( c \right)\\
&  =  \left( \widetilde{\omega}^+_n \mathrm{Col}^+(c \cdot \mathrm{loc} \ b),  \widetilde{\omega}^-_n \mathrm{Col}^-(c \cdot \mathrm{loc} \ b) \right) \\
&  =  \left( \widetilde{\omega}^+_n \mathrm{Col}^+(\mathrm{loc} \ \mathbf{z}_{\mathrm{Kato}}),  \widetilde{\omega}^-_n \mathrm{Col}^-(\mathrm{loc} \ \mathbf{z}_{\mathrm{Kato}}) \right) \\
&  =  \left( \widetilde{\omega}^+_n L^+_p(\mathbb{Q}_{\infty}, f), \widetilde{\omega}^-_n L^-_p(\mathbb{Q}_{\infty}, f)\right) .
\end{align*}

\section{Vanishing of $\mathrm{Err}_n$} \label{sec:vanishing_coker_g_n}
This section is entirely devoted to prove the following proposition in order to obtain Theorem \ref{thm:main_theorem_3}.
We keep all the assumptions in $\S$\ref{sec:non-CM} in this section.
\begin{prop} \label{prop:vanishing_of_coker_g_n}
If
\begin{enumerate}
\item[(fineNF)] $\mathrm{Sel}_{0}(\mathbb{Q}_{\infty}, E[p^{\infty}])^\vee$ has no nontrivial finite 
$\Lambda$-submodule, and
\item[(\textrm{{\cyr SH}})] if $\mathrm{char}_{\Lambda} \left( \mathrm{Sel}_{0}(\mathbb{Q}_{\infty}, E[p^{\infty}])^\vee \right) \subseteq \left( \Phi_{n}(1+X) \right)$, then $\mathrm{rk}_{\mathbb{Z}} E(\mathbb{Q}_{n})>\mathrm{rk}_{\mathbb{Z}} E(\mathbb{Q}_{n-1})$
(if $n=0$, then $\Phi_{0}(1+X)=X$ and this inequality means $\mathrm{rk}_{\mathbb{Z}} E(\mathbb{Q})>0$),
\end{enumerate}
then $\mathrm{coker} \ g_n = 0$.
\end{prop}
\subsection{Reduction}
We recall the following exact sequence (\cite[(Proof of) Proposition 3.4]{kurihara-pollack})
\begin{equation} \label{eqn:h1_loc_sel_sel_0-sequence}
\xymatrix{
0 \ar[r] & \dfrac{\mathbb{H}^1_{\mathrm{loc}}(T)}{\mathrm{loc} \ \mathbb{H}^1_{\mathrm{glob}}(T)} \ar[r] & \mathrm{Sel}(\mathbb{Q}_{\infty}, E[p^\infty])^\vee \ar[r] & \mathrm{Sel}_0(\mathbb{Q}_{\infty}, E[p^\infty])^\vee \ar[r] & 0 .
}
\end{equation}
By taking the $\Gamma_n$-coinvariant of the above sequence, we have the exact sequence
\begin{equation} \label{eqn:C_n-sequence}
\xymatrix{
0 \ar[r] & C_n \ar[r] &  \dfrac{\mathrm{H}^1(\mathbb{Q}_{n,p}, T)}{\mathrm{im} \ \mathbb{H}^1_{\mathrm{glob}}(T)} \ar[r] & \left( \mathrm{Sel}(\mathbb{Q}_{\infty}, E[p^\infty])^\vee \right)_{\Gamma_n} \ar[r] & \mathrm{Sel}_0(\mathbb{Q}_{n}, E[p^\infty])^\vee \ar[r] & 0 
}
\end{equation}
where 
$$C_n := \mathrm{coker} \left( 
 \mathrm{Sel}(\mathbb{Q}_{\infty}, E[p^\infty])^{\vee ,\Gamma_n} \to \mathrm{Sel}_0(\mathbb{Q}_{\infty}, E[p^\infty])^{\vee ,\Gamma_n}
  \right)$$
 and
$\mathrm{im} \ \mathbb{H}^1_{\mathrm{glob}}(T):=\mathrm{im} \left( \mathbb{H}^1_{\mathrm{glob}}(T) \to \mathrm{H}^1(\mathbb{Q}_{n,p}, T)\right)$.
We also have an exact sequence
\[
\xymatrix{
0 \ar[r] & \mathcal{Z}_n \ar[r] &  \dfrac{\mathrm{H}^1(\mathbb{Q}_{n,p}, T)}{E(\mathbb{Q}_{n,p})\otimes \mathbb{Z}_p} \ar[r] &  \mathrm{Sel}(\mathbb{Q}_{n}, E[p^\infty])^\vee  \ar[r] & \mathrm{Sel}_0(\mathbb{Q}_{n}, E[p^\infty])^\vee \ar[r] & 0 
}
\]
from Sequence (\ref{eqn:splitting_sequence}) and the definition of $\mathcal{Z}_n$.

Let $\pi_{\mathrm{glob},n} : \mathrm{H}^1(\mathbb{Q}_{n,p}, T) \to \dfrac{\mathrm{H}^1(\mathbb{Q}_{n,p}, T)}{\mathrm{im} \ \mathbb{H}^1_{\mathrm{glob}}(T)}$ be the natural projection and
 $$\widetilde{C}_n :=\pi^{-1}_{\mathrm{glob},n} (C_n) \subseteq \mathrm{H}^1(\mathbb{Q}_{n,p}, T)$$ be the inverse image of $C_n$ with respect to $\pi_{\mathrm{glob},n}$, and it obviously contains $\mathrm{im} \ \mathbb{H}^1_{\mathrm{glob}}(T)$.
 
Considering the following commutative diagram
\[
\xymatrix{
0 \ar[r] & \widetilde{C}_n \ar@{->>}[d] \ar[r] & \mathrm{H}^1(\mathbb{Q}_{n,p}, T) \ar@{->>}[d] \ar[rd] & E(\mathbb{Q}_{n,p}) \otimes \mathbb{Z}_p \ar[d] \\
0 \ar[r] & C_n \ar[r] &  \dfrac{\mathrm{H}^1(\mathbb{Q}_{n,p}, T)}{\mathrm{im} \ \mathbb{H}^1_{\mathrm{glob}}(T)} \ar[r] & \left( \mathrm{Sel}(\mathbb{Q}_{\infty}, E[p^\infty])^\vee \right)_{\Gamma_n} \ar@{->>}[d] \ar[r] & \mathrm{Sel}_0(\mathbb{Q}_{n}, E[p^\infty])^\vee \ar[r] \ar@{=}[d] & 0 \\
0 \ar[r] & \mathcal{Z}_n \ar[r] &  \dfrac{\mathrm{H}^1(\mathbb{Q}_{n,p}, T)}{E(\mathbb{Q}_{n,p})\otimes \mathbb{Z}_p} \ar[r] &  \mathrm{Sel}(\mathbb{Q}_{n}, E[p^\infty])^\vee  \ar[r] & \mathrm{Sel}_0(\mathbb{Q}_{n}, E[p^\infty])^\vee \ar[r] & 0 ,
}
\]
it is observed that $\widetilde{C}_n$ surjects $\mathcal{Z}_n$ under the natural quotient map $$\mathrm{H}^1(\mathbb{Q}_{n,p}, T) \twoheadrightarrow \dfrac{\mathrm{H}^1(\mathbb{Q}_{n,p}, T)}{E(\mathbb{Q}_{n,p})\otimes \mathbb{Z}_p} $$
since
$$\left( \dfrac{\mathrm{H}^1(\mathbb{Q}_{n,p}, T)}{\widetilde{C}_n} \right) / \left( E(\mathbb{Q}_{n,p})\otimes \mathbb{Z}_p \right)
=
\left(  \dfrac{\mathrm{H}^1(\mathbb{Q}_{n,p}, T)}{E(\mathbb{Q}_{n,p})\otimes \mathbb{Z}_p} \right) / \mathcal{Z}_n$$
as a subgroup of $\mathrm{Sel}(\mathbb{Q}_{n}, E[p^\infty])^\vee$.

Consider the composition of surjective maps
\[ 
\xymatrix{
\widetilde{C}_n \ar@{->>}[r] & \mathcal{Z}_n \ar@{->>}[r] & \mathrm{coker} \ g_n
}
\]
and then it factors through $C_n$ by definition.
Let $$C'_n := C_n \cap \mathrm{im} \left( E(\mathbb{Q}_{n,p}) \otimes \mathbb{Z}_p \to \dfrac{\mathrm{H}^1(\mathbb{Q}_{n,p}, T)}{\mathrm{im} \ \mathbb{H}^1_{\mathrm{glob}}(T)} \right) \subseteq \dfrac{\mathrm{H}^1(\mathbb{Q}_{n,p}, T)}{\mathrm{im} \ \mathbb{H}^1_{\mathrm{glob}}(T)} .$$
Then Sequence (\ref{eqn:C_n-sequence}) and the following exact sequence
{ \scriptsize
\[
\xymatrix{
0 \ar[r] & \mathrm{coker} \ g_n \ar[r] &  \dfrac{\mathrm{H}^1(\mathbb{Q}_{n,p}, T)}{E(\mathbb{Q}_{n,p})\otimes \mathbb{Z}_p + \mathrm{im} \ \mathbb{H}^1_{\mathrm{glob}}(T)} \ar[r] &  \mathrm{Sel}(\mathbb{Q}_{n}, E[p^\infty])^\vee  \ar[r] & \mathrm{Sel}_0(\mathbb{Q}_{n}, E[p^\infty])^\vee \ar[r] & 0 
}
\]
}
show that
$C'_n = \mathrm{ker} \left( C_n \to \mathrm{coker} \ g_n \right)$.
Thus, we have the exact sequence
\begin{equation} \label{eqn:criterion}
\xymatrix{
0 \ar[r] & C'_n \ar[r]^-{\varphi_n} & C_n  \ar[r] & \mathrm{coker} \ g_n \ar[r] & 0 .
}
\end{equation}
We can easily observe the following statement.
\begin{prop}
The following statements are equivalent:
\begin{enumerate}
\item $\mathrm{coker} \ g_n = 0$.
\item  $\varphi_n : C'_n \to C_n$ is an isomorphism.
\item All the classes in $C_n$ lie in $\overline{E(\mathbb{Q}_{n,p}) \otimes \mathbb{Z}_p} := \mathrm{im} \left( E(\mathbb{Q}_{n,p}) \otimes \mathbb{Z}_p \to \dfrac{\mathrm{H}^1(\mathbb{Q}_{n,p}, T)}{\mathrm{im} \ \mathbb{H}^1_{\mathrm{glob}}(T)} \right)$.
\end{enumerate}
\end{prop}
In particular, if $C_n = 0$, then $\mathrm{coker} \ g_n = 0$. 

From now on, we prove Proposition \ref{prop:vanishing_of_coker_g_n} using induction on $n$.

\subsection{When the rank does not grow}
In this subsection, we assume that
\begin{assu}
$\Phi_n(1+X)$ does not divide a generator of $\mathrm{char}_{\Lambda}\left( \mathrm{Sel}_0(\mathbb{Q}_\infty, E[p^\infty])^\vee\right)$.
\end{assu}
If $n=0$, then
$$\mathrm{Sel}_0(\mathbb{Q}_{\infty}, E[p^\infty])^{\vee, \Gamma} = 0.$$
Thus by the definition of $C_0$, we have $C_0=0$, which implies that $\mathrm{coker} \ g_0=0$.

Now we suppose $n > 0$.
\begin{lem} \label{lem:rank_does_not_grow_case}
Let $M$ be a finitely generated torsion $\Lambda$-module with no non-trivial finite $\Lambda$-submodule.
Suppose that $\Phi_n(1+X)$ does not divide a generator of $\mathrm{char}_{\Lambda}(M)$.
Then $M^{\Gamma_{n-1}} = M^{\Gamma_{n}}$.
\end{lem}
\begin{proof}
We may assume $M = M^{\Gamma_n}$, i.e. $\omega_n M = 0$.
Using the structure theorem for finitely generated $\Lambda$-modules, $M$ is a submodule of $M'$ of finite index with
$$M' \simeq \bigoplus_{i=1}^m \Lambda / f_i \Lambda .$$ 
Since $\omega_n M = 0$, we also have $\omega_n M' = 0$. It shows that each $f_i$ divides $\omega_n$. 
Since $\Phi_n(1+X)$ does not divide ${ \displaystyle \mathrm{char}_{\Lambda}(M ') = ( \prod_{i=1}^m f_i ) }$, each $f_i$ is prime to $\Phi_n(1+X)$.
Thus, each $f_i$ divides $\omega_{n-1} = \omega_n / \Phi_n(1+X)$ and then $\omega_{n-1} M' = 0$.
Hence, $\omega_{n-1}M = 0$.
\end{proof}
By Lemma \ref{lem:rank_does_not_grow_case}, we have
$$\mathrm{Sel}_0(\mathbb{Q}_\infty, E[p^\infty])^{\vee,\Gamma_{n-1}} \simeq \mathrm{Sel}_0(\mathbb{Q}_\infty, E[p^\infty])^{\vee,\Gamma_{n}} .$$
Note that Assumption (fineNF) is used here.
Thus, the natural map $C_{n-1} \to C_n$ is surjective.
Then we have the following commutative diagram
\[
\xymatrix{
C_n \ar@{->>}[r] & \mathrm{coker} \ g_n \\
C_{n-1} \ar@{->>}[r] \ar@{->>}[u] & \mathrm{coker} \ g_{n-1} \ar[u]
}
\]
and $\mathrm{coker} \ g_{n-1} = 0$ by the induction hypothesis. 
Thus, the lower horizontal map $C_{n-1} \to \mathrm{coker} \ g_{n-1}$ becomes the zero map and then 
$\mathrm{coker} \ g_{n}=0$.
\subsection{When the rank grows}
In this subsection, we assume that
\begin{assu}
$\Phi_n(1+X)$ divides a generator of $\mathrm{char}_{\Lambda}\left( \mathrm{Sel}_0(\mathbb{Q}_\infty, E[p^\infty])^\vee\right)$.
\end{assu}
If $n=0$, then Assumption (\textrm{{\cyr SH}}) implies $\mathrm{rk}_{\mathbb{Z}}E(\mathbb{Q})>0$, so the natural map
\begin{equation} \label{eqn:surjectivity_0}
E(\mathbb{Q}) \otimes \mathbb{Q}_p / \mathbb{Z}_p \to E(\mathbb{Q}_p) \otimes \mathbb{Q}_p / \mathbb{Z}_p
\end{equation}
is surjective.

Let
$$ M_{\Phi_n} := M / \Phi_n(1+X) M $$
where $M$ is a $\Lambda_n$-module and $\Phi_n(1+X) \in \Lambda_n$.

Now we suppose $n > 0$.
By Assumption (\textrm{{\cyr SH}}), we have
$$\mathrm{rk}_{\mathbb{Z}} E(\mathbb{Q}_n) > \mathrm{rk}_{\mathbb{Z}} E(\mathbb{Q}_{n-1}) .$$
Then the map
\begin{equation} \label{eqn:surjectivity_n}
\left( E(\mathbb{Q}_n) \otimes \mathbb{Q}_p / \mathbb{Z}_p \right)_{\Phi_n} \to \left( E(\mathbb{Q}_{n,p}) \otimes \mathbb{Q}_p / \mathbb{Z}_p \right)_{\Phi_n}
\end{equation}
is surjective.

First, we explicitly write down the connecting map
$$C_n  \to \frac{\mathrm{H}^1(\mathbb{Q}_{n,p}, T)}{\mathrm{im} \ \mathbb{H}^1_{\mathrm{glob}}(T)} $$
which is obtained by taking the $\Gamma_n$-coinvariant of Sequence (\ref{eqn:h1_loc_sel_sel_0-sequence}).
Let
$$[f] \in C_n = \mathrm{coker} \left( \mathrm{Sel}(\mathbb{Q}_{\infty}, E[p^\infty])^{\vee, \Gamma_n} \to \mathrm{Sel}_0(\mathbb{Q}_{\infty}, E[p^\infty])^{\vee, \Gamma_n} \right)$$
with a representative $f \in  \mathrm{Sel}_0(\mathbb{Q}_{\infty}, E[p^\infty])^{\vee}[\omega_n] \subseteq \mathrm{Sel}_0(\mathbb{Q}_{\infty}, E[p^\infty])^\vee $.
Via Sequence (\ref{eqn:h1_loc_sel_sel_0-sequence}), we lift $f$ to $\widetilde{f} \in \mathrm{Sel}(\mathbb{Q}_{\infty}, E[p^\infty])^\vee $.
Also, since $\omega_n \widetilde{f}$ maps to $\omega_nf = 0$ in Sequence (\ref{eqn:h1_loc_sel_sel_0-sequence}), we have
$$\omega_n \widetilde{f} \in \frac{\mathbb{H}^1_{\mathrm{loc}}(T)}{ \mathrm{loc} \ \mathbb{H}^1_{\mathrm{glob}}(T)} \subseteq \mathrm{Sel}(\mathbb{Q}_{\infty}, E[p^\infty])^\vee  .$$
This means that there exists an element $P \in \mathbb{H}^1_{\mathrm{loc}}(T)$ such that
$$\omega_n \widetilde{f}(x) = \langle P, j(x)\rangle$$
for any $x \in \mathrm{Sel}(\mathbb{Q}_{\infty}, E[p^\infty])$
where $j : \mathrm{Sel}(\mathbb{Q}_{\infty}, E[p^\infty]) \to \mathrm{H}^1(\mathbb{Q}_{\infty, p}, E[p^\infty])$ is the natural localization map and $\langle -,-\rangle$ is the local Tate pairing between $\mathbb{H}^1_{\mathrm{loc}}(T)$ and $\mathrm{H}^1(\mathbb{Q}_{\infty, p}, E[p^\infty])$.
Putting $P_n := P \Mod{\omega_n} \in \mathrm{H}^1(\mathbb{Q}_{n,p}, T)$, we have the following diagram
\[
\xymatrix{
\mathbb{H}^1_{\mathrm{loc}}(T) \ar@{->>}[r] \ar@{->>}[d] & \mathrm{H}^1(\mathbb{Q}_{n,p},T)\ar@{->>}[d] & P \ar@{|->}[r] \ar@{|->}[d] & P_n \ar@{|->}[d]\\
\frac{\mathbb{H}^1_{\mathrm{loc}}(T)}{ \mathrm{loc} \ \mathbb{H}^1_{\mathrm{glob}}(T)}
 \ar@{->>}[r] & \frac{\mathrm{H}^1(\mathbb{Q}_{n,p},T)}{ \mathrm{im} \ \mathbb{H}^1_{\mathrm{glob}}(T)} & \omega_n \widetilde{f} \ar@{|->}[r] & \overline{P_n} := \omega_n \widetilde{f} \Mod{\omega_n}
}
\]
Note that 
$\overline{P_n} = \omega_n \widetilde{f} \Mod{\omega_n} \in \frac{\mathrm{H}^1(\mathbb{Q}_{n,p},T)}{ \mathrm{im} \ \mathbb{H}^1_{\mathrm{glob}}(T)}$ is not necessarily zero since $\widetilde{f}$ may not be contained in $\frac{\mathbb{H}^1_{\mathrm{loc}}(T)}{ \mathrm{loc} \ \mathbb{H}^1_{\mathrm{glob}}(T)}$.
To sum up, the map $$C_n  \to \frac{\mathrm{H}^1(\mathbb{Q}_{n,p}, T)}{\mathrm{im} \ \mathbb{H}^1_{\mathrm{glob}}(T)} $$ is defined by
$[f] \mapsto \overline{P_n} $.

Now we prove that $P_n \in E(\mathbb{Q}_{n,p}) \otimes \mathbb{Z}_p$ in $\mathrm{H}^1(\mathbb{Q}_{n,p}, T)$.
By the local Tate duality, it suffices to check
\begin{equation} \label{eqn:goal_when_rk_grows}
\langle P_n, E(\mathbb{Q}_{n,p}) \otimes \mathbb{Q}_p/\mathbb{Z}_p \rangle = 0.
\end{equation}

Suppose at first $n=0$. Since 
$$E(\mathbb{Q}) \otimes \mathbb{Q}_p/\mathbb{Z}_p \subseteq \mathrm{Sel}(\mathbb{Q}, E[p^\infty]) \subseteq 
\mathrm{Sel}(\mathbb{Q}_\infty, E[p^\infty])[\omega_0],$$
we know $\langle P_0, E(\mathbb{Q}) \otimes \mathbb{Q}_p/\mathbb{Z}_p \rangle = 0$. Since Map (\ref{eqn:surjectivity_0}) is 
surjective, we also have  
$ \langle P_0, E(\mathbb{Q}_p) \otimes \mathbb{Q}_p/\mathbb{Z}_p \rangle =0$. 

Next we consider the case $n>0$. Consider the exact sequence
\begin{equation} \label{eqn:mw_splitting}
{ \small \xymatrix{
0 \ar[r] & E(\mathbb{Q}_{n-1,p}) \otimes \mathbb{Q}_p/\mathbb{Z}_p
\ar[r] & E(\mathbb{Q}_{n,p}) \otimes \mathbb{Q}_p/\mathbb{Z}_p
\ar[r] & \left( E(\mathbb{Q}_{n,p}) \otimes \mathbb{Q}_p/\mathbb{Z}_p \right)_{\Phi_n}
\ar[r] & 0 .
} }
\end{equation}

By the induction hypothesis, we have
$\mathrm{coker}\ g_{n-1} = 0$.
Thus, the map in Sequence (\ref{eqn:criterion})
$$\varphi_{n-1} : C'_{n-1}  \to C_{n-1}$$
is an isomorphism and also $P_{n-1} := P_n \Mod{\omega_{n-1}}$ is contained in $E(\mathbb{Q}_{n-1,p}) \otimes \mathbb{Z}_p$.
Therefore, 
\begin{equation} \label{eqn:induction_hypothesis}
\left\langle P_{n},  E(\mathbb{Q}_{n-1,p}) \otimes \mathbb{Q}_p/ \mathbb{Z}_p  \right\rangle  = 0 .
\end{equation}
By Sequence (\ref{eqn:mw_splitting}), we only need to show that
\begin{equation} \label{eqn:Phi_n-quotient}
\left\langle P_n, \left( E(\mathbb{Q}_{n,p}) \otimes \mathbb{Q}_p/ \mathbb{Z}_p \right)_{\Phi_n} \right\rangle = 0 
\end{equation}
in order to prove (\ref{eqn:goal_when_rk_grows}).

As in the case $n=0$, $\omega_n \widetilde{f}$ vanishes on $E(\mathbb{Q}_n) \otimes \mathbb{Q}_p/\mathbb{Z}_p$ since 
$$E(\mathbb{Q}_n) \otimes \mathbb{Q}_p/\mathbb{Z}_p \subseteq \mathrm{Sel}(\mathbb{Q}_n, E[p^\infty]) \subseteq 
\mathrm{Sel}(\mathbb{Q}_\infty, E[p^\infty])[\omega_n].$$
Now we know that the homomorphism (\ref{eqn:surjectivity_n}) is surjective, so $\omega_n \widetilde{f}$ vanishes also 
on $(E(\mathbb{Q}_{n,p}) \otimes \mathbb{Q}_p/\mathbb{Z}_p)_{\Phi_n}$, and we get Equation (\ref{eqn:Phi_n-quotient}).

Sequence (\ref{eqn:mw_splitting}), Equation (\ref{eqn:induction_hypothesis}), and Equation (\ref{eqn:Phi_n-quotient}) complete the proof of Equation (\ref{eqn:goal_when_rk_grows}).

\section*{Acknowledgement}
C.K. thanks M.K. for introducing this subject; Robert Pollack for sharing his ideas, helpful discussion, and support; Takenori Kataoka for the suggestion to remove the Tamagawa number assumption; Shinichi Kobayashi for answering questions on $\pm$-Selmer groups; Jaehoon Lee for careful reading of an earlier version;  GilYoung Cheong, B.D. Kim, Myoungil Kim, Takahiro Kitajima, Meng Fai Lim, and Kazuto Ota for helpful discussions.
C.K. is partially supported by ``Overseas Research Program for Young Scientists" through Korea Institute for Advanced Study (KIAS) and by Basic Science Research Program through the National Research Foundation of Korea (NRF-2018R1C1B6007009).

M.K. thanks Ralph Greenberg and Robert Pollack for the discussion with them
on the subject of this paper (held more than 10 years ago). He also thanks C.K. heartily
for letting him come back to this interesting subject through many discussions and
questions.
M.K. is partially supported by JSPS Core-to-core program,
``Foundation of a Global Research Cooperative Center in Mathematics focused on Number
Theory and Geometry".

We both thank to anonymous referees for careful reading and comments, which improve the clarification of the exposition significantly.
\appendix
\section{Lemmas on Fitting ideals} \label{sec:fitting}
All the modules in this section are finitely presented over their base rings.
Let $R$ be a commutative ring with unity.
\begin{lem}[{\cite[1, Appendix]{mazur-wiles-main-conj}}] \label{lem:fitting_quotient}
Let $M \twoheadrightarrow N$ be a surjective map of $R$-modules.
Then $$\mathrm{Fitt}_{R} (M) \subseteq \mathrm{Fitt}_{R} (N).$$
\end{lem}

\begin{lem}[{\cite[9, Appendix]{mazur-wiles-main-conj}}] \label{lem:fitting_exact}
Let $0 \to M_1 \to M_2 \to M_3 \to 0$ be an exact sequence of $R$-modules.
Then
$$\mathrm{Fitt}_R (M_1) \cdot \mathrm{Fitt}_R (M_3) \subseteq \mathrm{Fitt}_R (M_2).$$
\end{lem}

\begin{lem}[{\cite[Theorem 22, Page 80]{northcott}}] \label{lem:fitting_exact_refined_proj_dim}
Let $0 \to M_1 \to M_2 \to M_3 \to 0$ be an exact sequence of $R$-modules.
Assume that $\mathrm{pd}_{R} M_3 \leq 1$.
Then
$$\mathrm{Fitt}_{R} (M_1) \cdot \mathrm{Fitt}_{R} (M_3) = \mathrm{Fitt}_{R} (M_2).$$
\end{lem}

\begin{lem} \label{lem:fitting_exact_refined_square}
Let $0 \to M_1 \to M_2 \to M_3 \to 0$ be an exact sequence of $R$-modules.
Assume that $M_3$ has presentation by a square matrix.
Then
$$\mathrm{Fitt}_R (M_1) \cdot \mathrm{Fitt}_R (M_3) = \mathrm{Fitt}_R (M_2).$$
\end{lem}
\begin{proof}
Denote a presentation matrix of $M_i$ by $A_i$ for $i = 1,2,3$. In other words,
\[
\xymatrix@R=0em{
R^{\oplus s} \ar[r]^{A_1} & R^{\oplus r} \ar[r] & M_1 \ar[r] & 0 \\
R^{\oplus (s+m)} \ar[r]^{A_2} & R^{\oplus (r+m)} \ar[r] & M_2 \ar[r] & 0 \\
R^{\oplus m} \ar[r]^{A_3} & R^{\oplus m} \ar[r] & M_3 \ar[r] & 0 
}
\]
with $r \leq s$. Then we have
$$
A_2 =\left(
\begin{array}{c | c}
A_1  & * \\ \hline
 0 & A_3
\end{array}\right) \in \mathrm{M}_{(r+m) \times (s +m)}(R).
$$
Considering $(r+m) \times (r +m)$ minors of $A_2$, it is easy to see that the upper-triangular part $(*)$ of $A_2$ does not affect the determinants of the minors. Thus, the conclusion follows.
\end{proof}
\begin{rem}
If $R=\mathbb{Z}[G]$ with a finite abelian group $G$, for example, and $M$ is torsion with $\mathrm{pd}_{R}M \leq 1$, then $\mathrm{Fitt}_{R} (M)$ is a principal ideal generated by a non-zero divisor.
Thus, Lemma \ref{lem:fitting_exact_refined_proj_dim} is slightly stronger than Lemma \ref{lem:fitting_exact_refined_square} for this case.
\end{rem}
\begin{lem}[{\cite[4, Appendix]{mazur-wiles-main-conj}}]  \label{lem:fitting-ideals-base-change}
Let $M$ be a finitely presented $R$-module.
If $I \subset R$ is an ideal, then 
$$\mathrm{Fitt}_{R/I}(M /IM) = \pi \left( \mathrm{Fitt}_R(M) \right)$$
where $\pi :  A \to A/I$ is the natural quotient map.
\end{lem}

The following lemma is the key to replace characteristic ideals by Fitting ideals.
\begin{lem} \label{lem:char-to-fitt}
Let $M$ be a finitely generated torsion $\Lambda$-module.
Assume that $M$ has no nontrivial finite $\Lambda$-submodule.
Then
$$ \mathrm{char}_{\Lambda}(M) = \mathrm{Fitt}_{\Lambda}(M) .$$
\end{lem}
\begin{proof}
Though several proofs of this lemma are known, we want to give here a new proof.
If $M$ has no nontrivial finite $\Lambda$-submodule, then $\mathrm{depth}(M) = 1$, i.e. there exists an element $x \in \Lambda$ such that the multiplication by $x$ map on $M$ is injective.
By Auslander-Buchsbaum formula, we have
$$\mathrm{pd}_{\Lambda}(M) + \mathrm{depth}(M) = \mathrm{depth}(\Lambda)$$
with $\mathrm{depth}(M) = 1$, and $\mathrm{depth}(\Lambda) = 2$. Thus, $\mathrm{pd}_{\Lambda}(M) = 1$.
This shows that $M$ is the cokernel of a $\Lambda$-homomorphism 
$f: \Lambda^n \to \Lambda^n$. Then both the characteristic ideal 
and the Fitting ideal of $M$ are generated by $\mathrm{det} ( f )$, and 
we get the conclusion.
See also \cite[Proposition 2.1]{wingberg-1985} and \cite[Lemma 1.3.3 and Proposition 1.3.4]{taleb-thesis}.
\end{proof}

\begin{rem}
This lemma is an enhanced version of \cite[page 327--328, Appendix]{mazur-wiles-main-conj} removing the $\mu=0$ assumption. (cf. \cite[$\S$1.1]{kurihara-fitting}.)
\end{rem}

\begin{lem} \label{lem:fitting_sub_free}
Let $M$ and $N$ be $\Lambda_n$-modules.
We assume that $M$ and $N$ have no finite $\Lambda$-torsion submodule provided that we regard $M$ and $N$ as $\Lambda$-modules.
If $N \subseteq M$, then
$$\mathrm{Fitt}_{\Lambda_n} ( M ) \subseteq  \mathrm{Fitt}_{\Lambda_n} ( N ) .$$
\end{lem}
\begin{proof}
We regard $M$ and $N$ as $\Lambda$-modules. Then we have
\[
\xymatrix{
\mathrm{Fitt}_{\Lambda} ( M ) = \mathrm{char}_{\Lambda} ( M ) & \mathrm{Fitt}_{\Lambda} ( N ) = \mathrm{char}_{\Lambda} ( N ) 
}
\]
by Lemma \ref{lem:char-to-fitt}.
Consider the exact sequence of $\Lambda$-modules
\[
\xymatrix{
0 \ar[r] & N \ar[r] & M \ar[r] & M/N \ar[r] & 0
}
\]
Then we have
\begin{align*}
\mathrm{Fitt}_{\Lambda} ( M ) & = \mathrm{char}_{\Lambda} ( M ) \\ 
 & =  \mathrm{char}_{\Lambda} ( N ) \cdot  \mathrm{char}_{\Lambda} ( M/N ) \\
 & =  \mathrm{Fitt}_{\Lambda} ( N ) \cdot  \mathrm{char}_{\Lambda} ( M/N ) \\
 & \subseteq  \mathrm{Fitt}_{\Lambda} ( N )
\end{align*}
By taking the quotient by $\omega_n$ with Lemma \ref{lem:fitting-ideals-base-change}, we have
$$\mathrm{Fitt}_{\Lambda_n} ( M ) \subseteq  \mathrm{Fitt}_{\Lambda_n} ( N ) .$$
 \end{proof}

\begin{lem} \label{lem:fitting_sub}
If $A \subset B$ as finitely generated $\Lambda_n$-modules, then 
$$\mathrm{Fitt}_{\Lambda_n}(B) \subseteq  \mathrm{Fitt}_{\Lambda_n}(A) .$$ 
\end{lem}
\begin{proof}
Consider the following two exact sequences as $\Lambda$ or $\Lambda_n$-modules with compatibility
\[
\xymatrix{
0 \ar[r] & B_{\mathrm{mft}} \ar[r]  & B \ar[r]  & B' \ar[r]  & 0 \\
0 \ar[r] & A_{\mathrm{mft}} \ar[r] \ar@{^{(}->}[u] & A \ar[r] \ar@{^{(}->}[u] & A' \ar[r] \ar@{^{(}->}[u] & 0 
}
\]
where $A_{\mathrm{mft}}$ and $B_{\mathrm{mft}}$ are the maximal finite torsion $\Lambda$-submodules of $A$ and $B$, respectively. Thus, $A'$ and $B'$ have no finite $\Lambda$-submodule and indeed have no finite $\Lambda_n$-submodule.
Then we have
$$ \mathrm{Fitt}_{\Lambda_n} (B') \subseteq \mathrm{Fitt}_{\Lambda_n} (A')$$  
by Lemma \ref{lem:fitting_sub_free}.
Also, by Mazur-Wiles \cite[Corollary to Proposition 3, Appendix, Page 328]{mazur-wiles-main-conj}, we have
$$\mathrm{Fitt}_{\Lambda_n}( B_{\mathrm{mft}} ) \subseteq \mathrm{Fitt}_{\Lambda_n}( A_{\mathrm{mft}} ) .$$
Then Lemma \ref{lem:fitting_exact_refined_proj_dim} and Lemma \ref{lem:fitting-ideals-base-change} show us that 
\begin{align*}
\mathrm{Fitt}_{\Lambda_n}(A) & = \mathrm{Fitt}_{\Lambda_n}( A_{\mathrm{mft}} ) \cdot \mathrm{Fitt}_{\Lambda_n} (A') \\
\mathrm{Fitt}_{\Lambda_n}(B) & = \mathrm{Fitt}_{\Lambda_n}( B_{\mathrm{mft}} ) \cdot \mathrm{Fitt}_{\Lambda_n} (B') .
\end{align*}
Thus, the conclusion follows.
\end{proof}

\bibliographystyle{amsalpha}

\bibliography{library}

\end{document}